\documentclass[12pt,leqno]{article}
\usepackage{amsmath,amssymb,amsthm,amsfonts,mathrsfs}
\usepackage{appendix}
\usepackage{array,enumitem,graphics,xcolor,yfonts,colonequals}
\usepackage{stmaryrd}
\usepackage[alphabetic]{amsrefs}
\usepackage[all,cmtip]{xy}
\usepackage{tikz-cd}
\usepackage[colorlinks,anchorcolor=blue,citecolor=blue,linkcolor=blue,urlcolor=blue,bookmarksopen=true]{hyperref}
\usepackage{comment}
\usepackage{mathtools}
\usepackage[margin=1in]{geometry}
\usepackage{titlesec}
\titleformat{\section}
  {\normalfont\large\bf}{\thesection}{1em}{}
\titleformat{\subsection}[runin]
  {\normalfont\normalsize\bf}{\thesubsection}{1em}{}

\usepackage{fancyhdr}
\fancypagestyle{titlepage}{}

\AtBeginDocument{%
	\def\MR#1{}
}

\newcommand{\bC}{\mathbb{C}}    %Complex numbers
\newcommand{\bH}{\mathbb{H}}    %Quaternion
\newcommand{\bP}{\mathbb{P}}    %Projective space
\newcommand{\bQ}{\mathbb{Q}}    %Rational numbers
\newcommand{\bR}{\mathbb{R}}    %Real numbers
\newcommand{\bZ}{\mathbb{Z}}    %Integers

\newcommand{\bL}{\mathbf{L}}
\newcommand{\br}{\mathbf{r}}
\newcommand{\cE}{\mathcal{E}}   %Coherent sheaf
\newcommand{\cH}{\mathcal{H}}   %Upper half plane
\newcommand{\cO}{\mathcal{O}}   %Structure sheaf
\newcommand{\Amp}{\mathrm{Amp}}         %Ample cone
\newcommand{\Cl}{\mathrm{Cl}}
\newcommand{\Coh}{\mathrm{Coh}} %Category of coherence sheaves

\newcommand{\Db}{\mathrm{D^b}}  %Bounded derived category
\newcommand{\End}{\mathrm{End}}
\newcommand{\Ext}{\mathrm{Ext}}
\newcommand{\id}{\mathrm{id}}
\newcommand{\Nef}{\mathrm{Nef}}
\newcommand{\Plov}{\mathrm{Plov}}
\newcommand{\pol}{\mathrm{pol}}   %categorical polynomial entropy
\newcommand{\PSL}{\mathrm{PSL}}         %Modular group
\newcommand{\real}{\mathrm{Re}}     %Real part
\newcommand{\red}{\mathrm{red}}
\newcommand{\Rot}{\mathrm{Rot}}
\newcommand{\rot}{\mathrm{rot}}
\newcommand{\sh}{\mathrm{sh}}
\newcommand{\SL}{\mathrm{SL}}           %Special linear group
\newcommand{\Sp}{\mathrm{Sp}}
\newcommand{\td}{\mathrm{td}}
\newcommand{\tr}{\mathrm{tr}}       %Trace

\newcommand{\Aut}{\operatorname{Aut}}   %Automorphism group
\newcommand{\Hom}{\operatorname{Hom}}   %Hom-set
\newcommand{\NS}{\operatorname{NS}}     %Neron-Severi group
\newcommand{\bigw}{\bigwedge\nolimits^{\!\bullet}}

\newtheorem*{thm*}{Theorem}
\newtheorem*{prop*}{Proposition}
\newtheorem*{cor*}{Corollary}
\newtheorem*{notn*}{Notation}
\newtheorem{thm}{Theorem}[section]
\newtheorem{prop}[thm]{Proposition}
\newtheorem{cor}[thm]{Corollary}
\newtheorem{lemma}[thm]{Lemma}

\numberwithin{equation}{section}

\theoremstyle{definition}

\newtheorem{eg}[thm]{Example}

\newtheorem{rmk}[thm]{Remark}

\begin{document}
\title{Categorical Dynamics of Abelian Varieties}
\author{Yu-Wei Fan}
\date{}
\maketitle

\begin{abstract}
We establish several general results on the categorical dynamics of abelian varieties.
\begin{enumerate}[label=(\roman*)]
    \item We prove the exponential and polynomial Gromov--Yomdin equalities for every autoequivalence of the derived category of an arbitrary abelian variety. The exponential equality extends earlier results of Kikuta for elliptic curves and of Yoshioka for abelian surfaces and simple abelian varieties.
    \item We prove the categorical polynomial entropy of every autoequivalence is bounded above by $g^2$, where $g$ is the dimension of the abelian variety. This gives a categorical generalization of the polynomial log-volume growth of Lin--Oguiso--Zhang and recovers their upper bound as a direct consequence.
    \item We identify the shifting number of every autoequivalence of an abelian variety with the pullback of the normalized Barge--Ghys symplectic translation number on the universal cover of $\text{Sp}(4g,\mathbb{R})$.
    \item We prove that the set of reduced shifting numbers is finite if and only if the nef cone is rational polyhedral. Moreover, when this set is infinite, there exists an autoequivalence with transcendental shifting number. In particular, this gives a negative answer to a question of Dimitrov--Haiden--Katzarkov--Kontsevich concerning the algebraicity of categorical entropy functions.
\end{enumerate}
\end{abstract}

\setcounter{tocdepth}{3}
\tableofcontents

%--------------------%--------------------%--------------------
\section{Introduction}
\label{sec:intro}
Dimitrov--Haiden--Katzarkov--Kontsevich \cite{DHKK} introduced categorical entropy to measure the growth of iterates of an endofunctor of a triangulated category. Their construction associates to an autoequivalence $\Phi$ an entropy function $t\mapsto h_t(\Phi)$, in which the parameter $t$ records cohomological shifts. Subsequent work has developed numerical bounds, polynomial refinements, and asymptotic shifting invariants \cites{KikutaTakahashi,KST20,FanFuOuchi,FanFilip}. 

In this paper, for complex abelian varieties, we show that these are governed by the numerical action and a symplectic translation invariant. 
We determine the full categorical entropy function and its polynomial refinement, prove a sharp dimension bound, and characterize the arithmetic of shifting numbers in terms of the nef cone.

Categorical entropy recovers classical entropy in both algebraic and symplectic geometry. For an automorphism $f$ of a smooth projective complex variety, Kikuta--Takahashi \cite{KikutaTakahashi}*{Theorem~5.5} proved that $h_0(\bL f^*)=h_{\text{top}}(f)$. Similarly, for a pseudo-Anosov mapping class $\phi$ of a closed oriented surface of genus at least two, the induced autoequivalence of its split-closed, $\bZ/2$-graded derived Fukaya category has categorical entropy equal to the topological entropy of a pseudo-Anosov representative \cite{DHKK}*{Theorem~2.17}.

%--------------------%--------------------%--------------------
\medskip
\noindent\textbf{Entropy and numerical growth.}
The results of Gromov and Yomdin identify the topological entropy of a surjective holomorphic self-map of a compact K\"ahler manifold with the logarithm of the spectral radius of its action on cohomology \cites{Gromov,Yomdin}. Its categorical analogue asks whether
$$ 
h_0(\Phi) \overset{?}{=} \log \rho(N(\Phi)), 
$$ 
where $N(\Phi)$ is the action on the numerical Grothendieck group. This equality was conjectured by Kikuta--Takahashi \cite{KikutaTakahashi}*{Conjecture~5.3}. It holds for curves \cite{Kikuta}, for varieties with ample or anti-ample canonical bundle \cite{KikutaTakahashi}, and for abelian surfaces and simple abelian varieties \cite{Yoshioka}, but fails in general \cites{FanCYEntropy,OuchiSpherical,Mattei}.
Nevertheless, the lower bound $h_0(\Phi)\ge\log\rho(N(\Phi))$ holds for every autoequivalence of a smooth proper category \cite{KST20}*{Theorem~2.13}; see also \cite{FanFuOuchi}*{Proposition~4.3}.

To capture the growth that remains after the leading exponential term, the notion of categorical polynomial entropy $h_t^{\pol}$ was introduced in \cite{FanFuOuchi}. Whenever the categorical Gromov--Yomdin equality
$h_0(\Phi)=\log\rho(N(\Phi))$ holds, categorical polynomial entropy satisfies the corresponding lower bound $h_0^{\pol}(\Phi)\geq s(N(\Phi))$ \cite{FanFuOuchi}*{Proposition~4.3}. Here $s(T)$ denotes the polynomial growth rate of an invertible linear operator $T$; it is one less than the maximal size of a Jordan block associated with an eigenvalue of modulus $\rho(T)$.

From now on, let $A$ be a complex abelian variety of dimension $g\geq1$, and let $\Db(A)$ denote its bounded derived category of coherent sheaves. Write $\tau(\Phi)$ for the common upper and lower shifting number of an autoequivalence of $\Db(A)$; their coincidence was proved in \cite{FanShifting} and will be recalled below. Our first result establishes the exponential and polynomial Gromov--Yomdin equalities for arbitrary abelian varieties.

\begin{thm}
\label{thm:main-entropy}
For every $\Phi\in\Aut\Db(A)$ and every $t\in\bR$, one has
$$
h_t(\Phi)=\log\rho(N(\Phi))+t\tau(\Phi),
\qquad
h_t^{\pol}(\Phi)=s(N(\Phi)).
$$
\end{thm}

In particular, polynomial entropy is independent of $t$; we write it as
$h^{\pol}(\Phi)$.
The proof gives more than the two limiting equalities. 
For every fixed split generator $G$ and parameter $t$, we establish the uniform comparison
$$
\epsilon_t(G,\Phi G)\asymp_{A,G,t}
 e^{t\tau(\Phi)}\|N(\Phi)\|,
$$
where $\epsilon_t$ is the weighted total dimension of Hom-spaces and the comparison constants are independent of $\Phi$.

%--------------------%--------------------%--------------------
\medskip
\noindent\textbf{A sharp bound for polynomial entropy.}
Our second result gives a uniform upper bound for polynomial entropy depending only on the dimension of $A$.

\begin{thm}
\label{thm:main-poly-bound}
For every $\Phi\in\Aut\Db(A)$,
$$
h^{\pol}(\Phi)\leq g^2.
$$
Moreover, this bound is attained in every dimension by an autoequivalence of a power of an elliptic curve.
\end{thm}

The geometric counterpart is \emph{polynomial log-volume growth} studied by Cantat--Paris-Romaskevich \cite{CPR} and Lin--Oguiso--Zhang \cite{LOZ}. For a zero-entropy automorphism $f$ of a smooth projective variety $X$ of dimension $d$ and an ample class $H$, put
$$
\Delta_n(f,H)=\sum_{r=0}^{n-1}(f^r)^*H,
\qquad
\Plov(f)=\limsup_{n\to\infty}
\frac{\log\left(\Delta_n(f,H)^d\right)}{\log n}.
$$
This invariant is independent of the choice of $H$. Lin--Oguiso--Zhang proved the sharp bound $\Plov(f)\leq g^2$ for complex tori of dimension $g$ \cite{LOZ}*{Theorem~1.4}; Hu--Jiang subsequently extended this bound to arbitrary normal projective varieties \cite{HuJiang}*{Corollary~1.2}.

We relate this invariant directly to categorical dynamics.
If $L$ is ample on a smooth projective variety $X$ and $f$ has zero topological entropy, then Proposition~\ref{prop:plov-identification} gives
$$
h^{\pol}((-\otimes L)\circ \bL f^*)=\Plov(f).
$$
Thus Theorem~\ref{thm:main-poly-bound} recovers the log-volume bound for abelian varieties and extends it from these geometric autoequivalences to all derived autoequivalences.

%--------------------%--------------------%--------------------
\medskip
\noindent\textbf{Universality of shifting numbers.}
Theorem~\ref{thm:main-entropy} separates the entropy function into a numerical intercept and a shifting slope.
In \cite{FanFilip}, we introduced upper and lower shifting numbers, which describe the asymptotic slopes of the entropy function:
$$
 \tau^\pm(\Phi)=\lim_{t\to\pm\infty}\frac{h_t(\Phi)}{t}.
$$
For abelian varieties, these two numbers coincide \cite{FanShifting}*{Theorem~3.1}; we denote their common value by $\tau(\Phi)$. 
Our next theorem identifies their common value for abelian varieties with a classical symplectic invariant.

Let $\gamma_A\colon\Aut\Db(A)\twoheadrightarrow U(A)$ be Orlov's homomorphism \cite{Orlov}*{Theorem~4.14}l; see Section~\ref{subsec:2.2}.
The duality pairing and complex structure equip $W_A=H_1(A\times\widehat A,\bR)$ with a symplectic form $\Omega_A$.
The homological action gives a faithful representation
$$
r_A\colon U(A)\hookrightarrow\Sp(W_A,\Omega_A);
$$
see Section~\ref{subsec:5.2}.
Write $\pi\colon\widetilde{\Sp}(W_A,\Omega_A) \longrightarrow\Sp(W_A,\Omega_A)$ for the universal covering homomorphism, with central kernel $\bZ\cong\langle z\rangle$.
By \cite{BIW10}*{Corollary~7.10}, there is a unique continuous homogeneous quasimorphism
$$
\widetilde{\Rot}\colon \widetilde{\Sp}(W_A,\Omega_A)\longrightarrow\bR
$$
satisfying $\widetilde{\Rot}(z)=1$.
It is one-half of the Barge--Ghys symplectic translation quasimorphism \cite{BargeGhys}; see also \cite{BIW10}*{Remark~7.11(2)}.

\begin{thm}
\label{thm:main-universal-shifting}
There is a unique group homomorphism
$$
\alpha_A\colon\Aut\Db(A)\longrightarrow \widetilde{\Sp}(W_A,\Omega_A)
$$
satisfying $\pi\circ\alpha_A=r_A\circ\gamma_A$ and
$$
\tau(\Phi)=\widetilde{\Rot}(\alpha_A(\Phi)) \qquad \text{for every } \Phi\in\Aut\Db(A).
$$
\end{thm}

This extends Orlov’s symplectic description for powers of an elliptic curve without complex multiplication \cite{Orlov}*{Example 4.15}, and the Lie-group description of shifting numbers on a finite-index subgroup for abelian surfaces in \cite{FanFilip}*{Theorem 5.2.2}.
It also implies
$$
e^{2\pi i\tau(\Phi)}\in\overline{\bQ} \qquad \text{for every } \Phi\in\Aut\Db(A).
$$
Consequently, every shifting number is rational or transcendental by Gelfond--Schneider; see Corollary~\ref{cor:shifting-rational-or-transcendental}.

%--------------------%--------------------%--------------------
\medskip
\noindent\textbf{Shifting spectra and the nef cone.}
Composing with a shift changes $\tau$ by an integer. Define the \emph{reduced shifting spectrum} by
$$
\Sigma^{\red}(A)=\{\tau(\Phi)+\bZ:\Phi\in\Aut\Db(A)\}\subseteq\bR/\bZ.
$$
We say that $A$ satisfies \emph{Condition~$(\star)$} if it is isogenous to a product $B_1\times\cdots\times B_s$ of pairwise nonisogenous simple abelian varieties, each with Picard number one.

\begin{thm}
\label{thm:main-shifting-spectrum}
The following conditions are equivalent:
\begin{enumerate}[label=(\alph*)]
\item $\Sigma^\red(A)$ is finite;
\item\label{item:thm1.4-b} $A$ satisfies Condition~$(\star)$;
\item\label{item:thm1.4-c} $\Nef(A)\subseteq\NS(A)_\bR$ is rational polyhedral;
\item every shifting number of $\Phi\in\Aut\Db(A)$ is rational.
\end{enumerate}
Moreover, if these conditions fail, then $\Sigma^\red(A)$ is dense in $\bR/\bZ$, and some autoequivalence has transcendental shifting number.
\end{thm}

The equivalence of \ref{item:thm1.4-b} and \ref{item:thm1.4-c} is due to Bauer \cite{Bauer}. 
The transcendental case has a consequence for a question of Dimitrov--Haiden--Katzarkov--Kontsevich \cite{DHKK}*{Section~4.1}, who asked whether entropy graphs are generally algebraic over $\bQ$ in the coordinates $(e^t,e^{h_t})$.
For the affine entropy functions here, an irrational slope already rules out every polynomial relation, even with complex coefficients; see Lemma~\ref{lemma:entropy-graph-criterion}.

\begin{cor}
\label{cor:main-DHKK-algebraicity}
There exist a complex abelian variety $A$ and an autoequivalence $\Phi\in\Aut\Db(A)$ such that
$$
\{(e^t,e^{h_t(\Phi)}):t\in\bR\}
$$
is not contained in the zero locus of any nonzero polynomial in $\bC[x,y]$.
\end{cor}

%--------------------%--------------------%--------------------
\medskip
\noindent\textbf{Outline of the proofs.}
The uniform comparison underlying Theorem~\ref{thm:main-entropy} begins with Orlov's purity theorem \cite{Orlov}, together with the results of Mukai and de Jong--Olsson \cites{Mukai_shvb,deJongOlsson}, realizes the Fourier--Mukai kernel as a finite pushforward of a line bundle.
The Mumford--Kempf index theorem, Riemann--Roch, and polynomial interpolation then bound total cohomology by finitely many Euler characteristics, with constants uniform in the kernel (Proposition~\ref{prop:unif-cohomo-estimate}).
These Euler characteristics are fixed linear functionals of the
numerical action. 
A bounded comparison between kernel shifts and shifting numbers incorporates the parameter $t$ and yields both entropy equalities.

For the polynomial bound in Theorem~\ref{thm:main-poly-bound}, we use the spin representation of Golyshev--Lunts--Orlov \cite{GLO} to realize the cohomological Fourier--Mukai action, up to scalar and conjugacy, as the total exterior power of an operator on a $2g$-dimensional space. A linear-algebraic lemma bounds its largest Jordan block by $g^2+1$. Since the numerical action is an invariant subquotient of the cohomological action, Theorem~\ref{thm:main-entropy} yields $h^{\pol}(\Phi)\leq g^2$. Sharpness follows from the identification of $\Plov(f)$ with $h^{\pol}((-\otimes L)\circ \bL f^*)$ and the extremal examples of Lin--Oguiso--Zhang on powers of elliptic curves.

Theorem~\ref{thm:main-universal-shifting} builds on Polishchuk's normalized comparison between the categorical shift extension and the degree--logarithm extension \cite{Polishchuk14}. We construct an equivariant embedding $D_A=\NS(A)_\bR+i\Amp(A)$
into the Siegel space of $W_A=H_1(A\times\widehat A,\bR)$, under which the degree automorphy factor agrees with the symplectic determinant factor. 
This identifies the degree extension with the pullback of the symplectic universal cover and yields the categorical lift. 
A uniform comparison between the resulting logarithmic height and the coherent-kernel shift shows that their homogenizations coincide, giving the desired formula for $\tau$.

Finally, the proof of Theorem~\ref{thm:main-shifting-spectrum} uses the compatibility of Orlov's cocycle with isogenies and products. Under Condition~$(\star)$, commensurability and product decompositions reduce finiteness to simple abelian varieties of Picard number one. For such a factor, a scalar congruence subgroup of $\SL_2(\bZ)$ has finite index in its integral isometry group, and the corresponding rotation numbers take only finitely many values.

When Condition~$(\star)$ fails, a polarized isogeny factor $B$ admits a nontrivial Rosati-fixed totally real field $L$.
The embedding $\SL_2(L)\to U(B)_\bQ$ expresses the reduced shifting number as an integral multiple of the sum of the rotation numbers over the real embeddings of $L$. 
After clearing denominators, we construct an integral element with exactly one elliptic embedding and all other embeddings hyperbolic. 
The elliptic rotation angle is an irrational multiple of $\pi$, while the hyperbolic factors contribute zero rotation. Orlov's surjectivity then produces an autoequivalence with transcendental shifting number.

%--------------------%--------------------%--------------------
\medskip
\noindent\textbf{Note on related work.}
This paper has been in preparation for nearly a year. After becoming aware of Yoshida's preprint \cite{Yoshida26}, submitted to arXiv on September 23, 2026, the author decided to make the present results publicly available without further delay. Yoshida proves the exponential entropy equality in Theorem~\ref{thm:main-entropy} by a different method, based on Bridgeland stability conditions, whereas the proof given here uses more classical tools involving Fourier--Mukai kernels and a uniform cohomology estimate (see Proposition~\ref{prop:unif-cohomo-estimate}).

%--------------------%--------------------%--------------------
\medskip
\noindent\textbf{Organization of the paper.}
Section~\ref{sec:prelim} collects the preliminaries.
Section~\ref{sec:entropy} proves the uniform comparison and entropy equalities, and Section~\ref{sec:poly-upper} establishes the sharp polynomial bound.
Section~\ref{sec:universal-shifting} identifies shifting
numbers with symplectic translation.
Section~\ref{sec:shifting-spectrum} proves the classification of reduced shifting spectra.

%--------------------%--------------------%--------------------
\medskip
\noindent\textbf{Conventions.}
All varieties are defined over $\bC$. All autoequivalences are exact and $\bC$-linear.
The dual of an abelian variety $A$ is denoted by $\widehat A$.
For a smooth projective variety $X$, let
$$
\chi(E,F)=\sum_{k\in\bZ}(-1)^k\dim_\bC\Hom(E,F[k]), \qquad
N(X)=K_0(X)/\operatorname{rad}\chi.
$$
The numerical Grothendieck group $N(X)$ has finite rank, and the induced pairing on $N(X)_\bR$ is nondegenerate. 
An autoequivalence $\Phi$ induces an invertible operator $N(\Phi)$ on $N(X)_\bR$. 
We fix a norm on $N(X)_\bR$ and use the associated operator norm; its choice affects the estimates below only by constant factors.

For nonnegative quantities $a,b$, we write $a\asymp_C b$ if they are bounded above and below by positive constant multiples of one another, with the allowed dependence indicated by the subscript $C$. 
For an invertible operator $T$ on a finite-dimensional real or complex vector space, write $\rho(T)$ for its spectral radius and $s(T)$ for one less than the largest size of a Jordan block whose eigenvalue has modulus $\rho(T)$. 
Then
\begin{equation}
\label{eq:linear-growth}
 \|T^n\|\asymp_T\rho(T)^n n^{s(T)} \qquad (n\geq1).
\end{equation}
We also write $d_{\max}(T)$ for the largest size of \emph{any} Jordan block, so that $s(T)\leq d_{\max}(T)-1$.

%--------------------%--------------------%--------------------
\medskip
\noindent\textbf{Acknowledgment.}
The author is partially supported by the National Natural Science Foundation of China (Grant No.~12401053).

%--------------------%--------------------%--------------------
\section{Preliminaries}
\label{sec:prelim}

%--------------------%--------------------%--------------------
\subsection{Polarizations, adjoints, and orthogonal decompositions}
In this subsection, we recall some standard notation and basic facts about abelian varieties, most of which can be found in standard references such as \cite{BL04}.

For abelian varieties $A,B$, put
$$
\Hom^0(A,B)=\Hom(A,B)\otimes_\bZ \bQ, \qquad \End^0(A)=\Hom^0(A,A).
$$
We write $\widehat{f}\colon \widehat{B}\to\widehat{A}$ for the dual of a homomorphism $f\colon A\to B$, and extend this notation to rational and real homomorphisms.
A homomorphism $q\colon A\to\widehat{A}$ is \emph{symmetric} if $\widehat{q}=q$ under the canonical isomorphism between $A$ and its double dual.
For a line bundle $L$ on $A$, let $\varphi_L(a)=t_a^*L\otimes L^{-1}$. The map $[L]\mapsto\varphi_L$ identifies $\NS(A)_\bQ$ with the space of symmetric rational homomorphisms from $A$ to $\widehat{A}$, which we denoted by $\Hom^0(A,\widehat{A})^+$.

The polarization $\lambda=\varphi_H$ of an ample line bundle $H$ determines the Rosati involution on $\End^0(A)$: $x^\dagger\coloneqq \lambda^{-1}\widehat{x}\lambda$. Denote
$$
\End^0(A)^\dagger\coloneqq\{x\in\End^0(A):x^\dagger=x\}.
$$
There are isomorphisms
$$
\NS(A)_\bQ\xrightarrow{\sim}\Hom^0(A,\widehat{A})^+ \xrightarrow{\sim} \End^0(A)^\dagger, \qquad
[L] \mapsto \varphi_L \mapsto \lambda^{-1}\varphi_L.
$$

Put $V_A=H_1(A,\bR)$, and let $J$ be its complex structure.
Under the natural identification $H_1(\widehat{A},\bR)\cong V_A^\vee$, the dual complex structure is $J_{\widehat{A}}\alpha=-\alpha\circ J$ \cite{GLO}*{\S1}.

Following the notations in \cite{BL04}*{\S2}, for a line bundle $L$, there is an associated alternating form $E_L\colon V_A\times V_A\to \bR$ and an associated Hermitian form $H_L\colon V_A\times V_A\to\bC$:
$$
H_L(v,w)=E_L(Jv,w)+i E_L(v,w), \qquad (\varphi_L)_*(v)=E_L(v,-).
$$
When $L$ is ample, the form $H_L$ is positive definite, and the Rosati involution is the adjoint operator for it:
$$
H_L(xv,w)=H_L(v,x^\dagger w).
$$
For a real symmetric homomorphism $q\colon A\to\widehat{A}$, set
$$
H_q(v,w)=H_L((\lambda^{-1}q)v,w)
$$
where $\lambda$ is the polarization assoicated with the ample line bundle $L$.
Write $p_A(q)$ for its positive index, counted in complex dimension. This depends only on $q$, not on the polarization.

\begin{lemma}
\label{lem:isogeny-inertia-q}
Let $\mu\colon B\to A$ be an isogeny, and let $q_B\in\Hom^0(B,\widehat{B})^+$.
Then $q_A=\widehat{\mu}^{-1}q_B\mu^{-1}$ is symmetric and has the same inertia as $q_B$.
\end{lemma}

\begin{proof}
It is straightforward to check that $q_A$ is symmetric, and the associated Hermitian forms satisfy
$$
H_{q_B}(v,w)=H_{q_A}(\mu v,\mu w).
$$
The assertion follows because $\mu$ is an isomorphism on the tangent spaces.
\end{proof}

We will use the following polarization-orthogonal form of Poincar\'e's reducibiility theorem \cite{BL04}*{\S5.3}.
For an abelian subviarety $i\colon B\hookrightarrow A$, let $C$ be the identity component of $\ker(\widehat{i}\lambda_A)$.
Then the addition map $\mu\colon B\times C\to A$ is an isogeny and
$$
\widehat{\mu}\lambda_A\mu=\lambda_B\oplus\lambda_C,
$$
where $\lambda_B$ and $\lambda_C$ are restrictions of $\lambda_A$.
Iterating this construction gives an isogeny from a product of simple abelian varieties to $A$ under which $\lambda_A$ pulls back to the product of the induced polarizations.

\begin{lemma}
\label{lem:totally-real-L}
Every number field contained in $\End^0(A)^\dagger$ is totally real.
\end{lemma}

\begin{proof}
Let $L\subseteq\End^0(A)^\dagger$ be a number field.
Choose a primitive element $\alpha\in L$, and let $p_\alpha(t)\in\bQ[t]$ be its minimal polynomial.
Since the action of $L$ on $H_1(A,\bQ)$ is faithful, the minimal polynomial of $\alpha$ acting on $H_1(A,\bQ)$ is also $p_\alpha(t)$. Hence $p_\alpha(t)$ divides the characteristic polynomial of this action.

On the other hand, $\alpha\in\End^0(A)^\dagger$ implies that $\alpha$ acts on $H_1(A,\bQ)\otimes_\bQ\bR$ as a self-adjoint operator with respect to the polarization. 
Thus every root of $p_\alpha(t)$ is real.
Since $L=\bQ(\alpha)$, every embedding $L\hookrightarrow\bC$ has image in $\bR$, so $L$ is totally real.
\end{proof}

%--------------------%--------------------%--------------------
\subsection{Fourier--Mukai equivalences and Orlov's cocycle}
\label{subsec:2.2}
In this subsection, we recall Orlov's results on derived equivalences of abelian varieties \cite{Orlov}, together with Polishchuk's results showing that cocycles can be determined and extended from their restriction to big subsets \cite{Polishchuk14}.

For $F=\bQ$ or $\bR$, one can write an element of $\End(A\times\widehat{A})_F=\End(A\times\widehat{A})\otimes_\bZ F$ into a block matrix.
Following \cite{Orlov}, put
$$
u=
\begin{pmatrix}
    a & b \\ c & d
\end{pmatrix},
\qquad
\widetilde{u}=
\begin{pmatrix}
    \widehat{d} & -\widehat{b} \\ -\widehat{c} & \widehat{a}
\end{pmatrix},
$$
and define
$$
U(A)_F=\{u\in\End(A\times\widehat{A})_F: \widetilde{u}=u^{-1}\}, \qquad 
U(A)=U(A)_\bQ \cap \End(A\times\widehat{A}).
$$
Following \cite{Polishchuk14}, we also denote
$$
U(A)_F^0=\left\{
\begin{pmatrix}
    a & b \\ c & d
\end{pmatrix}
\in U(A)_F:
b \text{ is invertible}
\right\}.
$$

By \cite{Orlov}*{Proposition~3.2}, every equivalence between derived categories of abelian varieties is Fourier--Mukai, with its kernel of the form $K[s]$ for a nonzero coherent sheaf $K$ and an integer $s$.
Orlov construed an exact sequence \cite{Orlov}*{Theorem~4.14}
$$
1\to \bZ\oplus (A\times\widehat{A}) \to \Aut\Db(A) \xrightarrow{\gamma_A} U(A) \to 1.
$$
Then
$$
\widetilde{U}(A)\coloneqq\Aut\Db(A)/(A\times\widehat{A}).
$$
is a central extension of $U(A)$ by $\bZ$.

Let $\Phi\in\Aut\Db(A)$ be an autoequivalence with Fourier--Mukai kernel $K[s]\in\Db(A\times A)$. Write
$$
\sh(\Phi)\coloneqq s
$$
for the shift of the coherent Fourier--Mukai kernel.
This descends to $\widetilde{U}(A)$.
The lifts with zero shift define a section $\sigma_A\colon U(A)\to\widetilde{U}(A)$, and identify the extension with $U(A)\times\bZ$ with multiplication
$$
(U,m)(v,n)=(uv,m+n+c_A(u,v)).
$$
Here $c_A$ is the cocycle satisfies $c_A(u_\Phi,u_\Psi)=\sh(\Phi\Psi)-\sh(\Phi)-\sh(\Psi)$.

Orlov extends this cocycle to $U(A)_\bR$ \cite{Orlov}*{\S4}.
If $u_1,u_2\in U(A)^0_\bR$ and denote $b_{12}$ the upper-right block of $u_1u_2$, then
\begin{equation}
\label{eqn:Orlov-cocycle-formula}
c_A(u_1,u_2)=p_A(b_1^{-1}b_{12}b_2^{-1})-g.    
\end{equation}

We will repeatedly use Polishchuk's extension principle.
Following \cite{Polishchuk14}*{Definition~1.3.2}, a subset $B\subseteq G$ of a group is called \emph{big} if
$$
B^{-1}\cap B g_1
\cap B g_2\cap B g_3\neq\varnothing
\qquad\text{for all }g_1,g_2,g_3\in G.
$$

\begin{lemma}[{\cite{Polishchuk14}*{Lemma~1.3.3(ii)}}]
\label{lem:big-subset-extension}
Let $ B\subseteq G$ be big, and let $c,c'\colon G\times G\to Z$ be a pair of $2$-cocycles, where $Z$ is an abelian group with the trivial $G$-action.
Suppose
$$
c(b_1,b_2)=c'(b_1,b_2)
\qquad\text{for any }b_1,b_2\in B \text{ with } b_1b_2\in B.
$$
Then the associated central extensions $G_c$ and $G_{c'}$ admit a unique isomorphism
$$
i\colon G_c\xrightarrow{\sim}G_{c'}
$$
inducing the identity on $G$ and $Z$ and satisfying $i(b,0)=(b,0)$ for every $b\in B$.
\end{lemma}

In particular, if $c$ and $c'$ agree on $B\times B$, then they agree on $G\times G$.
Indeed, bigness implies $G= B B$: for $g\in G$, choose $x\in B^{-1}\cap B g^{-1}$ and write $g=x^{-1}(xg)$.
For $g=b_1b_2$ with $b_1,b_2\in B$, the equality $c(b_1,b_2)=c'(b_1,b_2)$ therefore forces the isomorphism $i$ to satisfy $i(g,0)=(g,0)$.
Comparing the multiplication laws gives $c=c'$.

\begin{rmk}
We now record the big subsets to which we will apply Lemma~\ref{lem:big-subset-extension}.
\begin{enumerate}[label=(\roman*)]
    \item\label{item:big-cell-i} For $F=\bQ$ or $\bR$, the subset $U(A)_F^0$ is big in $U(A)_F$ by \cite{Polishchuk14}*{Lemma~1.3.4}.
    \item For $F=\bQ$ or $\bR$, the subset
$$
\left\{
\begin{pmatrix}a&b\\c&d\end{pmatrix}
\in\mathrm{SL}_2(F):b\neq0
\right\}
$$
    is big in $\SL_2(F)$.
    This is a special case of \ref{item:big-cell-i}, for $A$ an elliptic curve without complex multiplication.
    \item Bigness is preserved under finite products.
    \item For every finite-index subgroup $\Gamma\subseteq U(A)$, the subset $\Gamma\cap U(A)_\bQ^0$ is big in $\Gamma$ \cite{Polishchuk14}*{Remark~1.4.2}.
\end{enumerate}
\end{rmk}

%--------------------%--------------------%--------------------
\subsection{Categorical entropy function of an autoequivalence}
In this subsection, we recall the notions of categorical entropy and categorical polynomial entropy, following \cites{DHKK,FanFuOuchi}.

An object $G$ is a split generator if it generates the category under shifts, cones, and direct summmands.
If $P$ is very ample on a smooth projective variety $X$ of dimension $d$, then
$$
\bigoplus_{i=k-d}^k P^{\otimes i}
$$
is a split generator of $\Db(X)$ for every $k\in\bZ$ \cite{OrlovGenerator}*{Theorem~4}.

For objects $E,F$ and $t\in\bR$, the complexity $\delta_t(E,F)$ is the infimum of $\sum_j e^{n_jt}$ over towers of distinguished triangles with successive factors $E[n_j]$ and final object containing $F$ as a direct summand \cite{DHKK}*{Definition~2.1}.
Put
$$
\epsilon_t(E,F)=\sum_{k\in\bZ} \dim\Hom(E,F[k])e^{-kt}.
$$
For a split generator $G\in\Db(X)$, 
\begin{equation}
\label{eqn:complexity-Ext}
\delta_t(G,E)\asymp_{G,t}\epsilon_t(G,E)
\end{equation}
uniformly in $E$; see \cite{DHKK}*{Section~2.3}.

\begin{lemma}
\label{lemma:2-generator-estimate}
Fix split generators $G,G'$ of $\Db(X)$ and $t\in\bR$. Then
$$
\epsilon_t(G,E)\asymp_{G,G',t} \epsilon_t(G',E), \qquad
\epsilon_t(G,\Phi G) \asymp_{G,G',t} \epsilon_t(G',\Phi G')
$$
for every object $E$ and every autoequivalence $\Phi$, respectively.
The same holds for $\delta_t$.
\end{lemma}

\begin{proof}
By \cite{DHKK}*{Proposition~2.2}, we have
$$
\delta_t(G,E)\leq\delta_t(G,G')\delta_t(G',E), \qquad
\delta_t(G,\Phi G)\leq\delta_t(G,G')\delta_t(G',\Phi G')\delta_t(G',G).
$$
Interchanging $G$ and $G'$ gives the reverse inequalities.
Applying (\ref{eqn:complexity-Ext}) then proves the corresponding statements for $\epsilon_t$.
\end{proof}

Categorical entropy and categorical polynomial entropy are defined by \cites{DHKK,FanFuOuchi}
$$
h_t(\Phi)\coloneqq \lim_{n\to\infty}\frac{\log\delta_t(G,\Phi^nG)}{n},
\qquad
h_t^\pol(\Phi)\coloneqq \limsup_{n\to\infty} \frac{\log\delta_t(G,\Phi^nG)-nh_t(\Phi)}{\log n}.
$$
The first limit exists, and both invariants are independent of $G$.
They can be equally computed using $\epsilon_t(G,\Phi^nG')$ for any two fixed split generators.

The following standard estimate underlies the numerical lower bounds in
\cite{KST20}*{Theorem~2.13} and \cite{FanFuOuchi}*{Proposition~4.3}.

\begin{lemma}
\label{lemma:2-entropy-lower-bound}
For every split generator $G\in\Db(X)$, there is $C_G>0$ such that
$$
\| N(\Phi) \| \leq C_G \epsilon_0(G,\Phi G)
$$
for every autoequivalence $\Phi$.
\end{lemma}

\begin{proof}
Choose objects $E_i$ whose numerical classes $e_i$ form a basis of $N(X)_\bR$.
By nondegeneracy of the Euler pairing, the linear map $T\mapsto(\chi(e_i,Te_j))_{i,j}$ on $\End(N(X)_\bR)$ is injective.
Hence, by equivalence of norms on a finite-dimensional vector space, there exists $C>0$ such that
$$
\| N(\Phi) \| \leq C \max_{i,j} |\chi(E_i,\Phi E_j)|\leq 
C\epsilon_0(G_1,\Phi G_1),
\qquad
G_1\coloneqq G\oplus\bigoplus_i E_i.
$$
Apply Lemma~\ref{lemma:2-generator-estimate} proves the claim.
\end{proof}

%--------------------%--------------------%--------------------
\subsection{Kernel shifts and shifting numbers}
In this subsection, we recall the shifting numbers for abelian varieties \cites{FanFilip,FanShifting} and identify them with the homogenized kernel shifts.

A function $q$ on a group is a quasimorphism of defect at most $D$ if $|q(xy)-q(x)-q(y)|\leq D$.
Its homogenization
$$
q^h(x)\coloneqq \lim_{n\to\infty} \frac{q(x^n)}{n}
$$
exists and satisfies $|q^h-q|\leq D$.

For nonzero objects $E,F$, let $\epsilon^+(E,F)$ and $\epsilon^-(E,F)$ be the maximum and minimum of the integers $k$ for which $\Hom(E,F[-k])\neq0$, whenever this set is nonempty. The \emph{shifting numbers} of an autoequivalence $\Phi$ are
$$
\tau^\pm(\Phi)=\lim_{n\to\infty}\frac{\epsilon^\pm(G,\Phi^n G)}{n};
$$
see \cite{FanFilip}.
They give the asymptotic linear growth of the entropy function $h_t(\Phi)$ as $t\to\pm\infty$; see \cite{ElaginLunts}*{Proposition~6.13} and \cite{FanFilip}*{Theorem~2.1.7}. For abelian varieties, we have $\tau^+(\Phi)=\tau^-(\Phi)$ \cite{FanShifting}*{Theorem~3.1}.
We recall a kernel estimate that proves this and will also be used below.

\begin{lemma}
\label{lemma-2-sh-tau-defect}
The function
$$
\sh\colon \Aut\Db(A) \to \bZ
$$
is a quasimorphism of defect at most $2g$.
Its homogenization is $\tau=\tau^+=\tau^-$, and
$$
|\sh(\Phi)-\tau(\Phi)|\leq 2g.
$$
\end{lemma}

\begin{proof}
Write the kernels of $\Phi$ and $\Psi$ as $K_\Phi[s_\Phi]$ and $K_\Psi[s_\Psi]$, respectively, with $K_\Phi,K_\Psi$ coherent. Let $p_{ij}\colon A^3\to A^2$ be the projections. The convolution of the unshifted kernels is
$$
\mathcal C
=
Rp_{13*}\bigl(
p_{12}^*K_\Psi\otimes^{\mathbf L}p_{23}^*K_\Phi
\bigr),
$$
which has cohomology in degrees $[-2g,g]$.
On the other hand, we have
$$
\mathcal C\cong K_{\Phi\Psi}[c],
\qquad
c=s_{\Phi\Psi}-s_\Phi-s_\Psi,
$$
where $K_{\Phi\Psi}$ is a nonzero coherent sheaf.
Its unique nonzero cohomology sheaf lies in degree $-c$. Thus
$$
-g\leq
\sh(\Phi\Psi)-\sh(\Phi)-\sh(\Psi)
\leq 2g,
$$
proving that $\sh$ has defect at most $2g$.

We now identify its homogenization. Take $G=\bigoplus_{i=0}^g P^{\otimes i}$ with $P$ very ample. If the kernel of $\Phi$ is $K[s]$, then for any two line bundles $L_1,L_2$,
$$
\Hom(L_1,\Phi L_2[-k])
\cong
H^{s-k}\bigl(
A\times A,K\otimes p_1^*L_2\otimes p_2^*L_1^\vee
\bigr).
$$
The right-hand side can be nonzero only when $0\leq s-k\leq 2g$.
Consequently, every nonzero term in $\epsilon_t(G,\Phi G)$
satisfies
\begin{equation}
\label{eqn:kernel-degree-window}
\sh(\Phi)-2g \leq k\leq \sh(\Phi).
\end{equation}
Applying this to $\Phi^n$ gives
$$
\frac{\sh(\Phi^n)}{n}-\frac{2g}{n}
\leq
\frac{\epsilon^\pm(G,\Phi^nG)}{n}
\leq
\frac{\sh(\Phi^n)}{n}.
$$
Letting $n\to\infty$, we obtain
$$
\tau^+(\Phi)=\tau^-(\Phi)
=\lim_{n\to\infty}\frac{\sh(\Phi^n)}{n}.
$$
Thus their common value $\tau(\Phi)$ is the homogenization of
$\sh$. The standard estimate for homogenization
therefore gives
$$
|\sh(\Phi)-\tau(\Phi)|\leq 2g,
$$
as required.
\end{proof}

The shifting number modulo $\bZ$ depends only on the Orlov element, which we called the \emph{reduced shifting number}
$$
\overline\tau_A\colon U(A)\longrightarrow\bR/\bZ,
\qquad
\overline\tau_A(u)=\tau(\Phi)+\bZ,
$$
where $\Phi$ is any autoequivalence with $u_\Phi=u$. 
For later use, we extend this definition to $U(A)_F$ for $F=\bQ,\bR$. 
Let $\tau_A$ denote the homogenization of the central coordinate $(u,m)\mapsto m$ on $\widetilde U(A)_F$, and set
$$
\overline\tau_A(u)=\tau_A(u,0)+\bZ.
$$
This agrees with the preceding definition on $U(A)$, since
$\tau_A(u_\Phi,\sh(\Phi))=\tau(\Phi)$.

%--------------------%--------------------%--------------------
\subsection{Poincar\'e translation and rotation numbers}
Consider the covering map
$$
\bR\longrightarrow\bR\bP^1,
\qquad
r\mapsto[\cos(\pi r):\sin(\pi r)].
$$
Let $\operatorname{Homeo}_{\bZ}^{+}(\bR)$ denote the group of increasing homeomorphisms $\widetilde h\colon\bR\to\bR$ satisfying $\widetilde h(r+1)=\widetilde h(r)+1$.
The \emph{Poincar\'e translation number} is
$$
\widetilde{\rot}(\widetilde h)
\coloneqq
\lim_{n\to\infty}
\frac{\widetilde h^{n}(r)-r}{n},
$$
which is independent of $r$.
For $h\in\SL_2(\bR)$, its \emph{rotation number} is
$$
\rot(h)
\coloneqq
\widetilde{\rot}(\widetilde h)+\bZ\in\bR/\bZ,
$$
where $\widetilde h$ is any lift of its action on $\bR\bP^1$.
This is independent of the chosen lift.

Let $\kappa\colon\SL_2(\bR)\times\SL_2(\bR)\to\bZ$ be the bounded $2$-cocycle characterized by
\begin{equation}
\label{eq:rank-one-cocycle}
\kappa(h_1,h_2)
=
\mathbf{1}_{\{b_{12}/(b_1b_2)>0\}}-1
\qquad (b_1b_2\neq0),
\end{equation}
where
$
h_i=
\begin{pmatrix}
a_i&b_i\\
c_i&d_i
\end{pmatrix}
$
and $b_{12}$ is the upper-right entry of $h_1h_2$.
Let
\[
G_\kappa\coloneqq\SL_2(\bR)\times_\kappa\bZ
\]
be the central extension with multiplication
\[
(h,m)(k,n)=(hk,m+n+\kappa(h,k)).
\]
The central-coordinate function $\ell_\kappa(h,m)=m$ is a
quasimorphism. Denote its homogenization by
\[
\tau_\kappa(\xi)
\coloneqq
\lim_{n\to\infty}\frac{\ell_\kappa(\xi^n)}{n},
\qquad \xi\in G_\kappa.
\]
Since $\tau_\kappa(h,m)=\tau_\kappa(h,0)+m$, its reduction
modulo $\bZ$ depends only on $h$.

\begin{lemma}
\label{lem:SL2-coycle-rotation}
The homogenized central coordinate for $\kappa$, reduced
modulo $\bZ$, is the projective rotation number:
\[
\tau_\kappa(h,0)+\bZ=\rot(h).
\]
It vanishes when $h$ has a real invariant line.
For an elliptic element
$h=\begin{pmatrix}a&b\\c&d\end{pmatrix}$
with eigenvalues $e^{\pm i\theta}$, where $0<\theta<\pi$,
its value is $-\theta/\pi$ if $b>0$, and $\theta/\pi$
if $b<0$, modulo $\bZ$.
In particular,
\[
\rot(\SL_2(\bZ))
=
\left\{0,\frac12,\frac13,\frac23\right\}.
\]
\end{lemma}

\begin{proof}
By (\ref{eq:rank-one-cocycle}), the bounded Euler cocycle defined by lifts with displacement at $1/2$ in $[0,1)$ differs from $\kappa$ by the coboundary of $\eta(h)=\mathbf{1}_{\{b(h)\neq0\}}$.
Thus $\kappa$ represents the pullback of the bounded integral Euler class with the orientation fixed above, and $\tau_\kappa(h,0)+\bZ=\rot(h)$
by \cite{BIW10}*{Proposition~7.5 and Remark~7.11(1)}; see also \cite{BargeGhys}*{\S B.1}.

If $h$ has a real invariant line, its projective action has a fixed point, so $\rot(h)=0$.
An elliptic element $h$ is conjugate in $\SL_2(\bR)$ to
$$
R_\varphi=
\begin{pmatrix}
\cos\varphi&-\sin\varphi\\
\sin\varphi&\cos\varphi
\end{pmatrix},
\qquad
\varphi=-\operatorname{sgn}(b)\theta.
$$
Since $R_\varphi$ acts on the covering coordinate by $r\mapsto r+\varphi/\pi$, conjugacy invariance gives the stated elliptic formula.

Finally, an elliptic element of $\SL_2(\bZ)$ has trace in $\{-1,0,1\}$, so its rotation number belongs to $\{1/2,1/3,2/3\}$; all other elements have rotation number zero.
All four values occur: writing
$$
S=\begin{pmatrix}0&-1\\1&0\end{pmatrix},
\qquad
R=\begin{pmatrix}0&-1\\1&1\end{pmatrix},
$$
the elements $I,S,R,R^{-1}$ have rotation numbers $0,1/2,1/3,2/3$, respectively.
\end{proof}

%--------------------%--------------------%--------------------
\subsection{Symplectic translation and rotation numbers}
\label{subsec:symplectic-translation}
Let $(W,\Omega)$ be a real symplectic vector space of dimension
$2n>0$.
We recall two related invariants: a real-valued translation number on the universal covering group of $\Sp(W,\Omega)$, and a rotation number on $\Sp(W,\Omega)$ with values in $\bR/\bZ$.

Let $\pi\colon\widetilde\Sp(W,\Omega)\longrightarrow \Sp(W,\Omega)$
be the universal covering homomorphism.
The central kernel of $\pi$ is infinite cyclic.
We choose its generator $z$ to be represented by a counterclockwise full rotation in one symplectic plane, acting trivially on its symplectic complement.

There is a unique continuous homogeneous quasimorphism
$$
\widetilde{\Rot}\colon\widetilde\Sp(W,\Omega)\longrightarrow\bR
$$
such that $\widetilde{\Rot}(z)=1$; see \cite{BIW10}*{Theorem~7.9 and Corollary~7.10}.
We call it the \emph{normalized symplectic translation number}.
Our normalization is one-half of the Barge--Ghys normalization; see \cite{BargeGhys}*{\S C.2} and \cite{BIW10}*{Remark~7.11(2)}.

Since a homogeneous quasimorphism is additive on commuting elements, centrality of $z$ gives
$$
\widetilde{\Rot}(z^m\widetilde M)
=
m+\widetilde{\Rot}(\widetilde M).
$$
Consequently, the \emph{normalized symplectic rotation number}
$$
\Rot(M)
\coloneqq
\widetilde{\Rot}(\widetilde M)+\bZ
\in\bR/\bZ,
\qquad \pi(\widetilde M)=M,
$$
is independent of the chosen lift.
Thus $\widetilde{\Rot}$ retains the integer winding information carried by the lift, while $\Rot$ depends only on $M$.

Choose a Lagrangian splitting $W=X\oplus X^\vee$ with
$$
\Omega((x,\alpha),(y,\beta))=\beta(x)-\alpha(y).
$$
Let $\cH(W)$ be the Siegel upper half-space of symmetric maps $Z\colon X_\bC\to X_\bC^\vee$ with $\operatorname{Im}Z>0$.
For
$M=\begin{pmatrix}A&B\\C&D\end{pmatrix}\in \Sp(W,\Omega)$,
the action on $\mathcal H(W)$ and its determinant factor are
\begin{equation}\label{eq:symplectic-action-det}
M\cdot Z=(C+DZ)(A+BZ)^{-1},
\qquad
j_{\det}(M,Z)=\det(A+BZ).
\end{equation}

The translation number can be described as the homogenization of an argument of this determinant.
Fix $Z_0\in\mathcal H(W)$, and represent $\widetilde M\in\widetilde\Sp(W,\Omega)$ by a path $(M_t)_{0\leq t\leq1}$ with $M_0=I$.
Since $j_{\det}(M_t,Z_0)\neq0$, there is a unique continuous argument $\theta(t)$ satisfying
$$
\theta(0)=0,
\qquad
e^{i\theta(t)}
=
\frac{j_{\det}(M_t,Z_0)}
{|j_{\det}(M_t,Z_0)|}.
$$
Put
$$
q_{Z_0}(\widetilde M)
\coloneqq-\frac{\theta(1)}{2\pi}.
$$
This is well-defined on homotopy classes of paths.
It is a continuous quasimorphism, and
$$
\widetilde{\Rot}(\widetilde M)
=
\lim_{N\to\infty}
\frac{q_{Z_0}(\widetilde M^N)}{N}.
$$
The limit is independent of $Z_0$.

In particular, the invariant is explicit on a maximal compact subgroup.
After choosing dual bases, the stabilizer of $iI$ consists of matrices
$$
k=
\begin{pmatrix}
A&B\\
-B&A
\end{pmatrix},
\qquad A+iB\in U(n),
$$
and
$$
\Rot(k)
=
-\frac{1}{2\pi}\arg\det(A+iB)+\bZ.
$$
More intrinsically, suppose that $k$ fixes $Z\in\mathcal H(W)$.
Then $k$ preserves the positive complex Lagrangian
$$
\Lambda_Z
=
\{(x,Zx):x\in X_\bC\}\subseteq W_\bC.
$$
Under the identification $X_\bC\cong\Lambda_Z$, its action is $A+BZ$. Hence
\begin{equation}\label{eq:symplectic-rotation-lagrangian}
\exp(2\pi i\Rot(k))
=
\det(k|_{\Lambda_Z})^{-1}.
\end{equation}

For an arbitrary $M\in \Sp(W,\Omega)$, consider its refined real
Jordan decomposition
$$
M=M_eM_hM_u,
$$
whose factors commute.
Here $M_e$ is semisimple with eigenvalues of modulus one,
$M_h$ is semisimple with positive real eigenvalues,
and $M_u$ is unipotent.
The rotation number depends only on the elliptic factor:
\begin{equation}\label{eq:symplectic-rotation-jordan}
\Rot(M)=\Rot(M_e),
\qquad
\Rot(M_h)=\Rot(M_u)=0;
\end{equation}
see \cite{BIW10}*{Proposition~7.8}.
Since $M_e$ is conjugate into a maximal compact subgroup, its rotation number is computed by the determinant formula above.

%--------------------%--------------------%--------------------
\section{Gromov--Yomdin equalities for entropy}
\label{sec:entropy}

The main result of this section is the following uniform comparison, which determines both the exponential and polynomial growth.

\begin{thm}
\label{thm:3-main-uniform-Hom-bound}
Let $A$ be a complex abelian variety, let $G$ be a split generator of $\Db(A)$, and fix $t\in\bR$. Then
$$
\epsilon_t(G,\Phi G)\asymp_{A,G,t} e^{t\tau(\Phi)}\| N(\Phi) \|
$$
for evert $\Phi\in\Aut\Db(A)$.
\end{thm}

The proof occupies the rest of this section. We first deduce Theorem~\ref{thm:main-entropy} from it.

\begin{proof}[Proof of Theorem~\ref{thm:main-entropy}]
Apply Theorem~\ref{thm:3-main-uniform-Hom-bound} to $\Phi^n$.
Homogeneity of $\tau$ and (\ref{eq:linear-growth}) give
$$
\epsilon_t(G,\Phi^n G) \asymp_{A,G,t,\Phi} e^{nt\tau(\Phi)}\rho(N(\Phi))^n n^{s(N(\Phi))}.
$$
Taking logarithm yields
$$
\log\epsilon_t(G,\Phi^nG)=n(\log\rho(N(\Phi))+t\tau(\Phi))+s(N(\Phi))\log n+O(1).
$$
Division by $n$ proves the entropy formula.
Subtracting the linear term and dividing by $\log n$ proves the polynomial formula.
\end{proof}

%--------------------%--------------------%--------------------
\subsection{Fourier--Mukai kernels as finite pushforwards}
In this subsection, we prove two lemmas that will be used to realize Fourier--Mukai kernels as pushforwards of line bundles along finite morphisms.

\begin{lemma}
\label{lemma:FM-kernel-point}
Let $K[s]$ be the Fourier--Mukai kernel of an equivalence $\Db(A)\simeq\Db(B)$ between abelian varieties, with $K\in\Coh(A\times B)$. Then
$$
\dim\Hom(K,K)=1, \qquad \dim\Ext^1(K,K)=\dim(A\times B).
$$
\end{lemma}

\begin{proof}
Derived equivalent abelian varieties have the same dimension; say $g=\dim(A)=\dim(B)$.
The external product equivalence
$$
\id_{\Db(A)}\boxtimes\Phi_{K[s]}\colon \Db(A\times A)\xrightarrow{\sim}\Db(A\times B)
$$
sends $\cO_{\Delta_A}$ to $K[s]$. Therefore
$$
\Ext^i(K,K)\cong\Ext^i(\cO_{\Delta_A},\cO_{\Delta_A}) \cong \text{HH}^i(A).
$$
The Hochschild--Kostant--Rosenberg isomorphism \cites{HKR62,Swan96} gives
$$
\text{HH}^0(A)\cong H^0(A,\cO_A), \qquad
\text{HH}^1(A)\cong H^1(A,\cO_A)\oplus H^0(A,T_A).
$$
Their dimensions are $1$ and $2g$, respectively.
\end{proof}

\begin{lemma}
\label{lemma:dJO}
Let $Y$ be a complex abelian variety and $S\in\Coh(Y)$ satisfying
$$
\dim\Ext^0(S,S)=1, \qquad \dim\Ext^1(S,S)\leq\dim(Y).
$$
Then there exists an abelian variety $Z$, a finite morphism $f\colon Z\to Y$, and a line bundle $M$ on $Z$ such that $S\cong f_*M$.
\end{lemma}

\begin{proof}
By \cite{deJongOlsson}*{Theorem~1.1}, we have $S\cong i_*E$, where $Z_0\subseteq Y$ is an abelian subvariety, $i\colon Z_0\hookrightarrow Y$ is its inclusion followed by a translation, and $E$ is a simple semihomogeneous vector bundle on $Z_0$. 
Mukai's theorem \cite{Mukai_shvb}*{Theorem~5.8} gives an isogeny $\pi\colon Z\to Z_0$ and a line bundle $M$ such that $E\cong\pi_*M$. 
Take $f=i\circ\pi$.
\end{proof}

%--------------------%--------------------%--------------------
\subsection{A dimension-uniform cohomology estimate}
In this subsection, we establish the key cohomological estimate needed for the proof of Theorem~\ref{thm:3-main-uniform-Hom-bound}.

\begin{prop}
\label{prop:unif-cohomo-estimate}
For each $N>0$, there is $C_N>0$ such that
$$
\sum_i h^i(Z,M)\leq C_N \max_{0\leq m\leq N}|\chi(Z,M\otimes H^{\otimes m})|
$$
for every abelian variety $Z$ of dimension at most $N$, every ample line bundle $H$, and every line bundle $M$ on $Z$.
\end{prop}

\begin{proof}
Put $d=\dim Z$. We may assume $H^\bullet(Z,M)\neq0$.
By \cite{BL04}*{Lemma~3.3.2 and Theorem~3.5.5}, there exist a quotient $\pi\colon Z\to Z'$ with connected kernel and a line bundle $M'$ on $Z'$ such that
\begin{enumerate}[label=(\roman*)]
    \item $M\cong\pi^*M'$;
    \item $M'$ has nonzero cohomology in precisely one degree $k$;
    \item writing $d'=\dim Z'$ and $r=d-d'$, one has
    $$
    h^{k+i}(Z,M)=\binom{r}{i}h^k(Z',M')
    \qquad (0\leq i\leq r),
    $$
    and all other cohomology groups of $M$ vanish.
\end{enumerate}
Consequently,
$$
\sum_i h^i(Z,M)=2^r|\chi(Z',M')|.
$$
Let $h=c_1(H)$ and $\ell'=c_1(M')$. By Hirzebruch--Riemann--Roch,
$$
P(m)\coloneqq\chi(Z,M\otimes H^{\otimes m})
=\frac{1}{d!}\int_Z(\pi^*\ell'+mh)^d
=\sum_{j=r}^d
\frac{m^j}{j!(d-j)!}\int_Z\pi^*(\ell'^{d-j})h^j,
$$
where the terms with $j<r$ vanish because $d-j>d'$.
The coefficient of $m^r$ is
$$
[m^r]P(m)
=\frac{1}{d'!r!}\int_Z\pi^*(\ell'^{d'})h^r
=\frac{1}{d'!}\int_{Z'}\ell'^{d'}\,
\pi_*\left(\frac{h^r}{r!}\right).
$$
Let $K=\ker\pi$. The class $\pi_*(h^r/r!)$ is the constant
$$
n_H\coloneqq\frac{1}{r!}\int_K c_1(H|_K)^r
=\chi(K,H|_K)\in\bZ_{>0},
$$
where the positivity follows from ampleness of $H$. Hence $[m^r]P(m)=n_H\chi(Z',M')$. In particular, we have
$$
|\chi(Z',M')|\leq |[m^r]P(m)|.
$$
Since the evaluation map
$$
\bR[m]_{\leq N}\longrightarrow\bR^{N+1},
\qquad Q\longmapsto(Q(0),\ldots,Q(N))
$$
is an isomorphism, there is a constant $C_N'>0$, depending only
on $N$, such that
$$
|[m^j]Q(m)|\leq C_N'\max_{0\leq m\leq N}|Q(m)|
\qquad (0\leq j\leq N)
$$
for every polynomial $Q$ of degree at most $N$.
Applying this to $P$, whose degree is $d\leq N$, gives
$$
\sum_i h^i(Z,M)
=2^r|\chi(Z',M')|
\leq 2^r|[m^r]P(m)|
\leq 2^NC_N'\max_{0\leq m\leq N}
|\chi(Z,M\otimes H^{\otimes m})|.
$$
Taking $C_N=2^NC_N'$ proves the proposition.
\end{proof}

\begin{cor}
\label{cor:cohomology-estimate}
Let $Y$ be an abelian variety of dimension $N$, let $H$ be ample, and let $S$ satisfy the hypotheses of Lemma~\ref{lemma:dJO}. Then
$$
\sum_i h^i(Y,S)\leq C_N \max_{0\leq m\leq N} |\chi(Y,S\otimes H^{\otimes m})|.
$$
\end{cor}

\begin{proof}
Write $S=f_*M$ as in Lemma~\ref{lemma:dJO}.
Since $f$ is finite, $f^*H$ is ample and $f$ has no higher direct images.
Applying Proposition~\ref{prop:unif-cohomo-estimate} to $Z$ with the ample line bundle $f^*H$, and using the projection formula, proves the claim.
\end{proof}

%--------------------%--------------------%--------------------
\subsection{Uniform comparison with the numerical action}
We prove Theorem~\ref{thm:3-main-uniform-Hom-bound} in this subsection.

\begin{prop}
\label{prop:entropy-upper-bound}
Let $A$ be an abelian variety of dimension $g$, and let $G$ be a split generator of $\Db(A)$. There is a constant $C_{A,G}>0$ such that
$$
\epsilon_0(G,\Phi G)\leq C_{A,G}\| N(\Phi)\|
$$
for every autoequivalence $\Phi$.
\end{prop}

\begin{proof}
By Lemma~\ref{lemma:2-generator-estimate}, we may take $G=\bigoplus_{i=0}^g P^{\otimes i}$ with $P$ very ample. Fix two summands $L_1,L_2$, and write the kernel of $\Phi$ as $K[s]$. Then
$$
\epsilon_0(L_1,\Phi L_2) = \sum_i h^i(A\times A,S), \qquad
S=K\otimes p_1^*L_2\otimes p_2^* L_1^\vee.
$$
Tensoring by a line bundle preserves self-Ext groups, so Lemma~\ref{lemma:FM-kernel-point} shows that $S$ satisfy the hypotheses of Lemma~\ref{lemma:dJO}. 

Fix an ample line bundle $H$ on $A$. Applying Corollary~\ref{cor:cohomology-estimate} with the ample line bundle $p_1^*H\otimes p_2^*H$ gives
\begin{align*}
\epsilon_0(L_1,\Phi L_2) & \leq C_{2g} \max_{0\leq m\leq 2g} |\chi(A\times A, K\otimes p_1^*(L_2\otimes H^m)\otimes p_2^*(L_1^\vee \otimes H^{m}))| \\
& = C_{2g} \max_{0\leq m\leq 2g} |\chi(L_1\otimes H^{-m},\Phi(L_2\otimes H^m))|.
\end{align*}
Each term in the last maximum is the absolute value of a fixed linear functional evaluated at $N(\Phi)$, and is therefore bounded by a constant times $\|N(\Phi)\|$.
Summing over the finitely many pairs of summands of $G$ proves the proposition.
\end{proof}

We now prove the main theorem of this section.

\begin{proof}[Proof of Theorem~\ref{thm:3-main-uniform-Hom-bound}]
Proposition~\ref{prop:entropy-upper-bound} and Lemma~\ref{lemma:2-entropy-lower-bound} give
$$
\epsilon_0(G,\Phi G) \asymp_{A,G} \| N(\Phi) \|.
$$
First take $G$ to be a direct sum of line bundles as above.
By (\ref{eqn:kernel-degree-window}), we have
$$
e^{t\sh(\Phi)-2g|t|} \epsilon_0(G,\Phi G)\leq \epsilon_t(G,\Phi G)\leq
e^{t\sh(\Phi)+2g|t|} \epsilon_0(G,\Phi G).
$$
By Lemma~\ref{lemma-2-sh-tau-defect}, replacing $\sh(\Phi)$ by $\tau(\Phi)$ changes each bound by a multiplicative factor between $e^{-2g|t|}$ and $e^{2g|t|}$. Combining these estimates gives
$$
\epsilon_t(G,\Phi G)
\asymp_{A,G,t}e^{t\tau(\Phi)}\|N(\Phi)\|.
$$
Finally, Lemma~\ref{lemma:2-generator-estimate} extends this comparison to every fixed split generator.
\end{proof}

%--------------------%--------------------%--------------------
\section{The sharp bound for polynomial entropy}
\label{sec:poly-upper}

%--------------------%--------------------%--------------------
\subsection{Upper bound for polynomial entropy}
We prove Theorem~\ref{thm:main-poly-bound} in this subsection.

We first realize $\Phi^H$ as a total exterior power, up to scalar and conjugacy, for every $\Phi\in\Aut\Db(A)$ of an abelian variety.

\begin{thm}
\label{thm:GLO-exterior-PhiH}
There is an isomorphism of complex vector spaces
$$
\beta_A\colon H^\bullet(A,\bC)\xrightarrow{\sim}\bigw H^{1,0}(A\times\widehat{A})
$$
such that, for every $\Phi\in\Aut\Db(A)$, there is a constant $c_\Phi\in\bC^\times$ with
$$
\beta_A \Phi^H \beta_A^{-1}=
c_\Phi\bigw
\left((u_\Phi^{-1})^*|_{H^{1,0}(A\times\widehat A)}\right).
$$
Here $u_\Phi\in U(A)\subseteq\Aut(A\times\widehat{A})$ denotes the Orlov element of $\Phi$.
\end{thm}

\begin{proof}
Put $V=H_1(A,\bC)$ and identify
$W=H_1(A\times\widehat{A},\bC)\cong V\oplus V^\vee$,
with the symmetric form
$$
Q((v,\alpha),(w,\beta))=\alpha(w)+\beta(v).
$$
Set
$$
E=H^{1,0}(A\times\widehat A)^\vee,
\qquad
F=H^{0,1}(A\times\widehat A)^\vee.
$$
Then $W$ admits two maximal isotropic decompositions:
$$
W=V\oplus V^\vee=E\oplus F.
$$
Indeed, the first is isotropic by the definition of $Q$, while
the second consists of the two eigenspaces of the product complex
structure, which preserves $Q$.

These decompositions give $\Cl(W,Q)$-module structures on
$$
\bigw V^\vee\cong H^\bullet(A,\bC),
\qquad
\bigw E^\vee
=\bigw H^{1,0}(A\times\widehat A).
$$
By \cite{GLO}*{Remark~3.6.1}, there is a Clifford-module isomorphism
$$
\beta_A\colon H^\bullet(A,\bC)
\xrightarrow{\sim}\bigw E^\vee.
$$
Fix such an isomorphism, independently of $\Phi$, and denote the
two Clifford actions by $c$ and $c'$, respectively.

Let $\Phi\in\Aut\Db(A)$ and put $u=(u_\Phi)_*$.
By \cite{GLO}*{Proposition~4.3.7},
$$
\Phi^Hc(w)(\Phi^H)^{-1}=c(uw)
\qquad(w\in W).
$$
Since $\beta_A$ intertwines the Clifford actions, the operator
$$
R_\Phi\coloneqq\beta_A\Phi^H\beta_A^{-1}
$$
satisfies
$$
R_\Phi c'(w)=c'(uw)R_\Phi
\qquad(w\in W).
$$
Moreover, $u_\Phi$ is holomorphic, so $u$ preserves both $E$ and $F$.

The pairing $Q$ identifies $F$ with $E^\vee$ via
$$
\kappa\colon F\xrightarrow{\sim}E^\vee,
\qquad
\kappa(f)(e)\coloneqq Q(e,f).
$$
In the Clifford module $\bigwedge^\bullet E^\vee$, elements of $E$
act by contraction, while $f\in F$ acts by exterior multiplication
by $\kappa(f)$. For every $e\in E$, 
$$
c'(ue)R_\Phi(1)=R_\Phi c'(e)(1)=0.
$$
Since $uE=E$, every contraction annihilates $R_\Phi(1)$.
Therefore, $R_\Phi(1)=c_\Phi1$ for some $c_\Phi\in\bC^\times$.
For $f_1,\ldots,f_r\in F$, we then obtain
\begin{align*}
R_\Phi\bigl(\kappa(f_1)\wedge\cdots\wedge\kappa(f_r)\bigr)
&=R_\Phi c'(f_1)\cdots c'(f_r)(1)\\
&=c'(uf_1)\cdots c'(uf_r)R_\Phi(1)\\
&=c_\Phi\,
\kappa(uf_1)\wedge\cdots\wedge\kappa(uf_r).
\end{align*}

Finally, let $a=u|_E$. Since $u$ preserves $Q$ \cite{GLO}*{Proposition~4.3.2}, we have
$$
\kappa(uf)(e)
=Q(e,uf)
=Q(u^{-1}e,f)
=\bigl((a^{-1})^\vee\kappa(f)\bigr)(e)
$$
for every $e\in E$ and $f\in F$. Thus the action of $u$ on $F$
corresponds under $\kappa$ to $(a^{-1})^\vee$ on $E^\vee$.
The preceding calculation gives
$$
R_\Phi=c_\Phi\bigw (a^{-1})^\vee.
$$
On $E^\vee=H^{1,0}(A\times\widehat A)$, the operator
$(a^{-1})^\vee$ is
$(u_\Phi^{-1})^*|_{H^{1,0}(A\times\widehat A)}$.
Consequently,
$$
\beta_A\Phi^H\beta_A^{-1}
=
c_\Phi\bigw
\left((u_\Phi^{-1})^*|_{H^{1,0}(A\times\widehat A)}\right).
$$
\end{proof}

Next, we prove a linear-algebraic lemma.

\begin{lemma}
\label{lem:linear-algebra-g^2}
Let $T$ be an invertible operator on a complex vector space of dimension $m$, with Jordan block sizes $k_1,\ldots,k_s$. Then
$$
d_{\max} \left(\bigw T\right)=1+\sum_{j=1}^s \left\lfloor \frac{k_j^2}{4}\right\rfloor \leq 1+ \left\lfloor \frac{m^2}{4} \right\rfloor.
$$
\end{lemma}

\begin{proof}
Write $V=\bigoplus_j V_j$, where $T|_{V_j}=\lambda_j \exp(N_j)$ and $N_j$ is nilpotent with one block of size $k_j$.
Consider multi-indices $\br=(r_1,\ldots,r_s)$ where $0\leq r_j\leq k_j$.
The decomposition
$$
\bigw V\cong \bigoplus_\br W_\br, \qquad 
W_\br \coloneqq \bigotimes_{j=1}^s \bigwedge\nolimits^{\! r_j} V_j
$$
is invariant under $\bigw T$.
On $W_\br$ its action is
$$
\left(\prod_j \lambda_j^{r_j}\right)\exp(D_\br), \qquad
D_\br=\sum_j \text{id}\otimes\cdots\otimes A_j\otimes\cdots\otimes\text{id},
$$
where $A_j\in\End(\bigwedge\nolimits^{\! r_j} V_j)$ is defined by
$$
A_j(x_1\wedge\cdots\wedge x_{r_j})=\sum_{a=1}^{r_j}x_1\wedge\cdots\wedge N_j x_a \wedge \cdots x_{r_j}
$$
which satisfies $\bigwedge\nolimits^{\! r_j} e^{N_j} = e^{A_j}$.

For each $V_j$, choose a basis $v_{j,0},\ldots,v_{j,k_j-1}$ with
$N_jv_{j,0}=0$ and $N_j v_{j,i}=v_{j,i-1}$.
Give a tensor of ordered wedges, with indices $i_{j,1}<\cdots<i_{j,r_j}$, height
$$
\sum_{j}\sum_{a=1}^{r_j} (i_{j,a}-(a-1)).
$$
Every nonzero term of $D_\br$ lowers height by one.
Heights range from zero to 
$$
M_\br=\sum_{j=1}^s r_j(k_j-r_j),
$$
so $D_\br^{M_\br+1}=0$.
It is straightforward to see that $D_\br^{M_\br}$ does not vanish on the tensor of top wedges.
The largest Jordan block on $W_\br$ therefore has size $M_\br+1$.
Maximizing over the multi-indices $\br$ proves the claim.
\end{proof}

We now prove the upper bound of categorical polynomial entropy.

\begin{proof}[Proof of Theorem~\ref{thm:main-poly-bound}]
Theorem~\ref{thm:main-entropy} and the fact that the numerical action $N(\Phi)$ is a subquotient of $\Phi^H$ give
$$
h^\pol(\Phi)=s(N(\Phi))\leq d_{\max} (\Phi^H)-1.
$$
By Theorem~\ref{thm:GLO-exterior-PhiH}, the last term is the largest Jordan block size minus one of a total exterior power of an invertible operator on a $2g$-dimensional space, and Lemma~\ref{lem:linear-algebra-g^2} bounds it by $g^2$.
Sharpness will be established in Example~\ref{eg:LOZ-sharp-g2}.
\end{proof}

%--------------------%--------------------%--------------------
\subsection{Polynomial log-volume growth}
Let $X$ be a smooth projective variety of dimension $d$, let $f\in\Aut(X)$, and let $H$ be ample. Define
$$
\Delta_n(f,H)=\sum_{r=0}^{n-1}(f^r)^*H, \qquad
\Plov(f)\coloneqq\limsup_{n\to\infty}\frac{\log \Delta_n(f,H)^d}{\log n}.
$$
This is the intersection-theoretic formulation of the \emph{polynomial log-volume growth} of $f$ \cites{CPR,LOZ}, which is independent of the choice of $H$.
Lin--Oguiso--Zhang proved that $\Plov(f)\leq g^2$ for every automorphism of zero topological entropy on a complex abelian variety of dimension $g$ \cite{LOZ}*{Theorem~1.4}.
The purpose of this subsection is to realize $\Plov(f)$ as the categorical entropy of an autoequivalence, therefore recovering the bound $\Plov(f)\leq g^2$ as a direct consequence of Theorem~\ref{thm:main-poly-bound}.

We first give an estimate needed to compare intersections with categorical complexity.

\begin{lemma}
\label{lem:for-plov-RR-estimate}
Let $X$ be a smooth projective variety of dimension $d$, and let $H$ be an ample divisor class.
There is $s_0>0$ such that, for $s\geq s_0$ and every nef real divisor class $D$,
$$
\frac{(sH+D)^d}{2d!} \leq \int_X e^{sH+D} \td(X) \leq \frac{2(sH+D)^d}{d!}.
$$
\end{lemma}

\begin{proof}
The class $H^p$ lies in the interior of the pseudoeffective cone \cite{FulgerLehmann}*{Theorem~1.4(2)}.
Put $\gamma_p(s)= \bigl(e^{sH}\operatorname{td}(X)\bigr)_p$. 
Since $s^{-p}\gamma_p(s) \to \frac{H^p}{p!}$,
for all sufficiently large $s$ one has
$$
\frac{s^pH^p}{2p!} \leq_{\mathrm{eff}} \gamma_p(s) \leq_{\mathrm{eff}} \frac{2s^pH^p}{p!}. 
$$
Choose a common threshold $s_0>0$ for all $0\leq p\leq d$.
Pairing these inequalities with $D^{d-p}/(d-p)!$ and summing over $p$ gives the desired claim.
\end{proof}

\begin{prop}
\label{prop:plov-identification}
Let $X$ be a smooth projective variety, let $f\in\Aut(X)$ have zero entropy, and let $L$ be ample. Then $\Phi=(-\otimes L)\circ \bL f^*$ satisfies
$$
h_0(\Phi)=0, \qquad h^\pol(\Phi)=\Plov(f).
$$
\end{prop}

\begin{proof}
Put $d=\dim X$, $H=c_1(L)$, and $H_n=(f^n)^*H$.
Choose an integer $m\geq s_0$ from Lemma~\ref{lem:for-plov-RR-estimate} such that $P=L^m$ is very ample and $P\otimes K_X^{-1}$ is ample. Consider the split generators
$$
G_-=\bigoplus_{i=1}^{d+1}P^{-i}, \qquad G_+=\bigoplus_{j=1}^{d+1}P^{j}.
$$
Write
$$
L_n=\bigotimes_{r=0}^{n-1} (f^r)^*L, \qquad
B_{ij,n}=P^i\otimes L_n\otimes (f^n)^*P^j.
$$
Then Kodaira vanishing gives
\begin{align*}
\epsilon_0(G_-,\Phi^n G_+) & = \sum_{i,j=1}^{d+1} \chi(X,B_{ij,n}) 
=\sum_{i,j}\int_X e^{miH+\Delta_n(f,H)+mjH_n}\td(X) \\
& \asymp_{X,L,m} \sum_{i,j=1}^{d+1} (miH+\Delta_n(f,H)+mjH_n)^d,
\end{align*}
where the last comparison follows from Lemma~\ref{lem:for-plov-RR-estimate}.

Choose $b>0$ such that $bH-f^*H$ is nef. Then
$$
H \leq_{\text{nef}} \Delta_n(f,H), \qquad
H_n \leq_{\text{nef}} b\Delta_n(f,H).
$$
It follows that
$$
\Delta_n(f,H) \leq_{\text{nef}}
miH+\Delta_n(f,H)+mjH_n \leq_{\text{nef}}
(mi+1+mjb) \Delta_n(f,H).
$$
Hence
$$
\epsilon_0(G_-,\Phi^n G_+) \asymp_{X,L,m,f} \Delta_n(f,H)^d.
$$
The Gromov--Yomdin theorem implies that $f^*$ on total cohomology has spectral radius one \cites{Gromov,Yomdin}. Thus the right hand side bounds polynomially in $n$, hence $h_0(\Phi)=0$.
Taking logarithms and dividing by $\log n$ proves the polynoimal entropy equality.
\end{proof}

\begin{cor}[\cite{LOZ}*{Theorem~1.4}]
Let $A$ be an abelian variety of dimension $g$, and let $f\in\Aut(A)$ be an automorphism of zero entropy. Then
$$
\text{Plov}(f)\leq g^2.
$$
\end{cor}

\begin{proof}
Choose an ample line bundle $L$. Apply Proposition~\ref{prop:plov-identification} and the bound in Theorem~\ref{thm:main-poly-bound} to $(-\otimes L)\circ \bL f^*$.
\end{proof}

%--------------------%--------------------%--------------------
\subsection{Pullbacks and sharpness}
In this subsection, we prove that the upper bound of polynomial entropy is sharp.

\begin{prop}
\label{prop:poly-ent-automorphism-upper-g^2/2}
Let $A$ be a complex abelian variety of dimension $g$, and let $f\in\Aut(A)$. Then
$$
h^\pol(\bL f^*)\leq \left\lfloor\frac{g^2}{2}\right\rfloor.
$$
\end{prop}

\begin{proof}
Let $T=f^*|_{H^{1,0}(A)}$ have Jordan-block sizes $k_1,\ldots,k_s$.
Under the canonical identification $H^\bullet(A,\bC)\cong\bigw H^1(A,\bC)$, the cohomological action $(\bL f^*)^H$ is identified with  $\bigw(T\oplus\overline{T})$.
The Jordan-block sizes of $T\oplus\overline{T}$ are $k_1,k_1,\ldots,k_s,k_s$.
Lemma~\ref{lem:linear-algebra-g^2} therefore gives
$$
d_{\max}((\bL f^*)^H)-1 = 2 \sum_j \left\lfloor \frac{k_j^2}{4} \right\rfloor
\leq \left\lfloor\frac{g^2}{2}\right\rfloor.
$$
Theorem~\ref{thm:main-entropy} and the fact that $N(\Phi)$ is a subquotient of $\Phi^H$ give
$$
h^\pol(\bL f^*) = s(N(\bL f^*))
\leq d_{\max}((\bL f^*)^H)-1,
$$
which proves the assertion.
\end{proof}

\begin{eg}
\label{eg:LOZ-sharp-g2}
Let $E$ be an elliptic curve and let
$$
f\colon E^g\to E^g, \qquad 
(x_1,\ldots,x_g) \mapsto (x_1,x_2+x_1,\ldots,x_g+x_{g-1}).
$$
Its action on $H^{1,0}(E^g)$ is unipotent with one Jordan block of size $g$.
By \cite{FanFuOuchi}*{Corollary~5.9} and Lemma~\ref{lem:linear-algebra-g^2},
$$
h^\pol(\bL f^*)=2\left\lfloor\frac{g^2}{4}\right\rfloor=\left\lfloor\frac{g^2}{2}\right\rfloor.
$$
Thus Proposition~\ref{prop:poly-ent-automorphism-upper-g^2/2} is sharp in every dimension.
For the same automorphism, \cite{LOZ}*{Example~6.4} gives $\Plov(f)=g^2$. Proposition~\ref{prop:plov-identification} then yields
$$
h^\pol((-\otimes L)\circ \bL f^*)=g^2
$$
for every ample line bundle $L$ on $E^g$.
This proves the sharpness in Theorem~\ref{thm:main-poly-bound}.
\end{eg}

%--------------------%--------------------%--------------------
\section{Universality of shifting numbers}
\label{sec:universal-shifting}
We prove Theorem~\ref{thm:main-universal-shifting} by comparing two quasimorphisms: the coherent-kernel shift and a logarithm of the symplectic determinant.

%--------------------%--------------------%--------------------
\subsection{The logarithm model and symplectic translation}
Let $(W,\Omega)$ have real dimension $2n>0$, with the conventions of Section~\ref{subsec:symplectic-translation}.
The determinant factor is nonvanishing on $\cH(W)$ and satisfies
$$
j_{\det}(M_1M_2,Z)=j_{\det}(M_1,M_2\cdot Z)j_{\det}(M_2,Z).
$$
Following \cite{Polishchuk14}*{\S2.3}, define
$$
\cE_{\det}(W)=\{(M,F)\in\Sp(W,\Omega)\times\cO(\cH(W)) \ | \ 
 e^{-2\pi iF(Z)}=j_{\det}(M,Z)\},
$$
with multiplication
\begin{equation}\label{eq:determinant-extension-law}
(M_1,F_1)(M_2,F_2)=(M_1M_2,F_1(M_2\cdot Z)+F_2(Z)).
\end{equation}
Then the projection $\cE_{\det}(W)\to\Sp(W,\Omega)$ is the universal covering homomorphism, with covering kernel generated by central generator $(I,1)$.

\begin{lemma}
\label{lem:qz0defectn/2}
For any $Z_0\in\cH(W)$, the function
$$
q_{Z_0}(M,F)=\text{Re} F(Z_0)
$$
is a continuous quasimorphism of defect at most $n/2$, and its homogenization is $\widetilde{\Rot}$ under the identification between $\cE_{\det}(W)$ and $\widetilde{\Sp}(W,\Omega)$.
\end{lemma}

\begin{proof}
We first claim that for every $\left(M=
\begin{pmatrix}
    A&B\\C&D
\end{pmatrix}
,F\right)\in\cE_{\det}(W)$ and $Z_0,Z_1\in\cH(W)$,
\begin{equation}\label{eq:determinant-sector}
|\real F(Z_1)-\real F(Z_0)|\leq n/2.
\end{equation}
When $B$ is invertible, $S=B^{-1}A$ is real symmetric. 
For $Z=X+iY$, $Y>0$, write
$$
j_{\det}(M,Z)=\det B\det Y\det(R+iI),
\qquad 
R=Y^{-1/2}(S+X)Y^{-1/2}.
$$
If $\lambda_1,\ldots,\lambda_n$ are the real eigenvalues of $R$, then an argument of this determinant is a fixed argument of $\det B$ plus $\sum_j\arg_{(0,\pi)}(\lambda_j+i)$. 
Its values lie in an interval of length $n\pi$, proving (\ref{eq:determinant-sector}). 
Matrices with invertible $B$ are dense, so continuity of logarithms gives the general case.

By (\ref{eq:determinant-extension-law}), the defect of $q_{Z_0}$ is
$$
\text{Re} F_1(M_2\cdot Z_0)-\text{Re} F_1(Z_0),
$$
which is bounded by $n/2$. 
Its homogenization is independent of $Z_0$ by (\ref{eq:determinant-sector}). It is continuous, since
$$
\left|q_{Z_0}^{h}(\xi)-\frac{q_{Z_0}(\xi^N)}N\right|\leq\frac{n}{2N}.
$$
Finally, it takes value $1$ on $(I,1)$. 
The uniqueness of the normalized continuous homogeneous quasimorphism \cite{BIW10}*{Corollary~7.10} identifies it with $\widetilde{\Rot}$.
\end{proof}

%--------------------%--------------------%--------------------
\subsection{The degree-determinant compatibility}
\label{subsec:5.2}
Put $V_A=H_1(A,\bR)$ and $W_A=V_A\oplus V_A^\vee$. 
Let $J_A=J\oplus J_{\widehat A}$, and define
$$
Q_A((v,\alpha),(w,\beta))=\alpha(w)+\beta(v),
\qquad 
\Omega_A(x,y)=-Q_A(J_Ax,y).
$$
Since $J_A$ is $Q_A$-orthogonal and $J_A^2=-I$, the form $\Omega_A$ is alternating and nondegenerate. 
The homological action of $U(A)_\bR$ preserves $Q_A$ and commutes with $J_A$ \cite{Orlov}*{Proposition~2.21}. 
It gives a faithful representation
$$
r_A\colon U(A)_\bR\to\Sp(W_A,\Omega_A).
$$

Consider the convex domain $D_A=\NS(A)_\bR+i\Amp(A)$. 
We also write $\omega$ for the corresponding complex symmetric homomorphism $A\to\widehat A$. 
Its homological realization satisfies $\omega_*(v)=E_\omega(v,-)$.
The action of $u\in U(A)_\bR$ on $\omega\in D_A$ is \cite{Polishchuk14}*{Section~1.4}
$$
u\cdot\omega=(c+d\omega)(a+b\omega)^{-1}, \qquad
u=\begin{pmatrix}a&b\\c&d\end{pmatrix}\in U(A)_\bR.
$$

\begin{prop}
\label{prop:degree-determinant}
There is a holomorphic equivariant embedding $\iota_A\colon D_A\hookrightarrow\cH(W_A)$ such that
$$
j_{\det}(r_A(u),\iota_A(\omega))=\deg(a+b\omega),
\qquad u=\begin{pmatrix}a&b\\c&d\end{pmatrix}\in U(A)_\bR.
$$
\end{prop}

\begin{proof}
Use the symplectic coordinates $(v,\alpha)\mapsto(-Jv,\alpha)$, in which $\Omega_A$ becomes the standard form, and set
$$
\iota_A(\omega)=Z_\omega=\omega_*J.
$$
Then
$$
Z_\omega(v)(w)=E_\omega(Jv,w)=E_\omega(Jw,v)=Z_\omega(w)(v).
$$
For $\omega=x+iH$ with $H$ ample, 
$$
(\text{Im} Z_\omega)(v)(v)=E_H(Jv,v)=H_H(v,v)>0.
$$
Thus $Z_\omega$ lies in the Siegel space. 
This gives the holomorphic embedding $\iota_A$.

In these coordinates, the matrix of $r_A(u)$ is
$$
\begin{pmatrix}a_*&-Jb_*\\c_*J&d_*\end{pmatrix}.
$$
Since $a_*$ and $b_*\omega_*$ commute with $J$, we obtain
\begin{align*}
r_A(u)\cdot Z_\omega
&=(c_*+d_*\omega_*)J
  (a_*-Jb_*\omega_*J)^{-1}\\
&=(c_*+d_*\omega_*)
  (a_*+b_*\omega_*)^{-1}J\\
&=\bigl((c+d\omega)(a+b\omega)^{-1}\bigr)_*J\\
&=Z_{u\cdot\omega}.
\end{align*}
This proves equivariance. Moreover,
$$
j_{\det}(r_A(u),Z_\omega)
=\det(a_*+b_*\omega_*)
=\deg(a+b\omega),
$$
where the determinant is taken on $V_{A,\bC}$.
The last equality follows by polynomial extension from the
identity
$\deg(f)=\det\bigl(f_*|_{H_1(A,\bQ)}\bigr)$
for isogenies $f\colon A\to A$.
\end{proof}

Write $\Delta_A(u)(\omega)=\deg(a+b\omega)$. 
Polishchuk's degree extension is
$$
U_A^\Delta=\{(u,f)\in U(A)_\bR\times\cO(D_A) \ |\  e^{-2\pi if}=\Delta_A(u)\},
$$
with the same logarithmic multiplication as (\ref{eq:determinant-extension-law}); see \cite{Polishchuk14}*{Section~2.3}. 
Proposition~\ref{prop:degree-determinant} identifies it with the pullback of the determinant cover:
\begin{equation}\label{eq:degree-pullback}
U(A)_\bR\times_{\Sp(W_A,\Omega_A)}\cE_{\det}(W_A)
\xrightarrow{\cong}U_A^\Delta,
\quad (u,(r_A(u),F))\longmapsto(u,F\circ\iota_A).
\end{equation}

%--------------------%--------------------%--------------------
\subsection{Universality of shifting numbers}
Let $B_A=U(A)\cap U(A)_\bQ^0$. 
We use the following integral form of Polishchuk's result.

\begin{prop}[\cite{Polishchuk14}*{Theorem~2.3.2}]
\label{prop:normalized-degree-comparison}
There is a unique isomorphism of central extensions
$$
I_A\colon\widetilde U(A)\xrightarrow{\sim}U_A^\Delta|_{U(A)}
$$
covering the identity on $U(A)$, sending the class of the shift functor $[m]$ to $(1,m)$ for every $m\in\bZ$, and satisfies
$$
I_A(\sigma_A(u))=(u,f_u)\qquad(u\in B_A),
$$
where $f_u$ is the logarithm normalized by
\begin{equation}\label{eq:big-log-normalization}
\lim_{t\to+\infty}\real f_u(itH)=-g/2
\end{equation}
for every ample class $H$.
\end{prop}

\begin{proof}[Proof of Theorem~\ref{thm:main-universal-shifting}]
Compose the quotient $\Aut\Db(A)\to\widetilde U(A)$ with $I_A$, the inverse of (\ref{eq:degree-pullback}), and projection to $\cE_{\det}(W_A)\cong\widetilde{\Sp}(W_A\Omega_A)$:
$$
\alpha_A\colon\Aut\Db(A)\to\widetilde{\Sp}(W_A,\Omega_A).
$$
It covers $r_A\circ\gamma_A$ and sends $[1]$ to $z$.

Fix $\omega_0\in D_A$ and pullback the logarithmic height $q_{Z_0}$ at $Z_0=\iota_A(\omega_0)$ to $\widetilde U(A)$; denote it by $q_A$. 
We claim
$$
|q_A(\xi)-\sh(\xi)|\leq3g
\qquad(\xi\in\widetilde U(A)).
$$
Both functions increase by $m$ under multiplication by $[m]$, so it suffices to take $\xi=\sigma_A(u)$. 
Write $u=u_1u_2$ with $u_1,u_2\in B_A$. 
The group laws give
$$
q_A(\sigma_A(u))=\real f_{u_1}(u_2\cdot\omega_0)
+\real f_{u_2}(\omega_0)-c_A(u_1,u_2).
$$
By Lemma~\ref{lem:qz0defectn/2} and (\ref{eq:big-log-normalization}), the two real parts lie in $[-\frac{3g}{2},\frac{g}{2}]$, and the cocycle $c_A(u_1,u_2)\in[-g,0]$.
This proves the claim.

Apply this bound to the class of $\Phi^N$ and divide by $N$. Lemmas~\ref{lemma-2-sh-tau-defect} and \ref{lem:qz0defectn/2} identify the two limits and give
$$
\tau(\Phi)=\widetilde{\Rot}(\alpha_A(\Phi)).
$$
Finally, any other lift with this equality has the form $\alpha'_A(\Phi)=z^{m(\Phi)}\alpha_A(\Phi)$ for an integer $m(\Phi)$. 
Additivity of $\widetilde{\Rot}$ on commuting elements gives
$$
\widetilde{\Rot}(\alpha'_A(\Phi))=m(\Phi)+\widetilde{\Rot}(\alpha_A(\Phi)).
$$
Hence $m(\Phi)=0$ for every $\Phi$, proving the uniqueness.
\end{proof}

\begin{cor}
\label{cor:shifting-rational-or-transcendental}
For every $\Phi\in\Aut\Db(A)$, 
$$
e^{2\pi i \tau(\Phi)} \in \overline{\bQ}.
$$
Consequently, every shifting number is either rational or transcendental over $\bQ$.
\end{cor}

\begin{proof}
Put $M=r_A(u_\Phi)$, and let $M_e$ be its elliptic factor.
By Theorem~\ref{thm:main-universal-shifting}, \cite{BIW10}*{Proposition~7.8} and (\ref{eq:symplectic-rotation-lagrangian}),
$$
e^{2\pi i\tau(\Phi)}
=\det(M_e|_\Lambda)^{-1},
$$
where $\Lambda\subseteq W_{A,\bC}$ is a positive Lagrangian preserved by $M_e$.
Since $M$ is integral, its eigenvalues $\lambda$ are algebraic, and so are the eigenvalues $\lambda/|\lambda|$ of $M_e$.
Thus the above determinant is algebraic.

If $\tau(\Phi)$ were algebraic irrational, the Gelfond--Schneider theorem would make $e^{2\pi i\tau(\Phi)}$, a value of $(-1)^{2\tau(\Phi)}$, transcendental \cite{Baker}*{Chapter~2}.
This proves the last assertion.
\end{proof}

%--------------------%--------------------%--------------------
\section{The reduced shifting spectrum}
\label{sec:shifting-spectrum}
In this section we prove Theorem~\ref{thm:main-shifting-spectrum}.
The case where Condition~$(\star)$ holds is treated in Sections~\ref{subsec:6.1}--\ref{subsec:6.2}, while the case where Condition~$(\star)$ fails is treated in Sections~\ref{subsec:6.3}--\ref{subsec:6.4}.

%--------------------%--------------------%--------------------
\subsection{Isogenies, products, and scalar embeddings}
\label{subsec:6.1}

We first prove the compatibility of Orlov cocycles under isogeny, products, and scalar embeddings.

\begin{lemma}
\label{lem:isogeny-cocycle}
Let $\mu\colon B\to A$ be an isogeny, and put
$$
f_\mu\coloneqq
\begin{pmatrix}
    \mu & 0 \\ 0  & \widehat{\mu}^{-1}
\end{pmatrix}
\colon B\times\widehat{B}\dashrightarrow A\times\widehat{A}.
$$
Then conjugation by $f_\mu$ identifies $U(B)_F$ with $U(A)_F$, for $F=\bQ,\bR$, and
$$
c_A(f_\mu u_1 f_\mu^{-1}, f_\mu u_2 f_\mu^{-1})=c_B(u_1,u_2).
$$
Consequently, $\overline{\tau}_A(f_\mu u f_\mu^{-1})=\overline{\tau}_B(u)$.
\end{lemma}

\begin{proof}
It is straightforward to verify that conjugation by $f_\mu$ gives an isomorphism between $U(B)_F$ and $U(A)_F$. Explicitly,
$$
f_\mu 
\begin{pmatrix}
a & b \\ c & d
\end{pmatrix}
f_\mu^{-1}=\begin{pmatrix}
    \mu a\mu^{-1} & \mu b \widehat{\mu} \\
    \widehat{\mu}^{-1} c \mu^{-1} & \widehat{\mu}^{-1} d \widehat{\mu}
\end{pmatrix}.
$$
This preserves the big cell: $f_\mu U(B)_F^0 f_\mu^{-1}=U(A)_F^0$.
For $u_1,u_2\in U(B)_F^0$, the cocycle $c_B(u_1,u_2)$ is determined by the symmetric homomorphism $q_B=b_1^{-1}b_{12}b_2^{-1}$ in (\ref{eqn:Orlov-cocycle-formula}); while the cocycle $c_A(f_\mu u_1 f_\mu^{-1}, f_\mu u_2 f_\mu^{-1})$ is determined by the symmetric homomorphism 
$$
q_A=( \mu b_1 \widehat{\mu})^{-1} ( \mu b_{12} \widehat{\mu}) ( \mu b_2 \widehat{\mu})^{-1}
=\widehat{\mu}^{-1} b_1^{-1}b_{12} b_2^{-1} \mu^{-1} = \widehat{\mu}^{-1} q_B \mu^{-1}.
$$
By Lemma~\ref{lem:isogeny-inertia-q}, the inertia of $q_A$ and $q_B$ coincide.
Therefore, $c_A$ and $c_B$ coincide on the big cells by (\ref{eqn:Orlov-cocycle-formula}), and extend to the whole $U(A)_F$ and $U(B)_F$ by Lemma~\ref{lem:big-subset-extension}.
The resulting isomorphism of central extensions preserves the integer coordinate.
Homogenization proves the assertion about reduced shifting numbers.
\end{proof}

\begin{lemma}
\label{lem:product-cocycle}
Let $B=\prod_{i=1}^s B_i$ be a product of abelian varieties.
For $u_i,v_i\in U(B_i)_\bR$,
$$
c_B\left(\bigoplus_i u_i, \bigoplus_i v_i\right) = 
\sum_i c_{B_i}(u_i,v_i), \qquad
\overline{\tau}_B\left(\bigoplus_i u_i\right)=
\sum_i \overline{\tau}_{B_i}(u_i).
$$
\end{lemma}

\begin{proof}
For $u_i,v_i$ lie in the big cells, the symmetric homomorphism in (\ref{eqn:Orlov-cocycle-formula}) is the direct sum of those of the factors.
Positive indices and dimensions are additive, giving the cocycle identity.
Then apply Lemma~\ref{lem:big-subset-extension} to extend the cocycle identity to every element.
Homogenization proves the assertion about reduced shifting numbers.
\end{proof}

For $F=\bQ$ or $\bR$, a polarization $\lambda\colon B\to\widehat B$
gives a scalar embedding
\begin{equation}
\label{eqn:scalar-embedding}
\iota_\lambda\colon \SL_2(F)\hookrightarrow U(B)_F,
\qquad
\begin{pmatrix}a&b\\c&d\end{pmatrix}
\mapsto
\begin{pmatrix}
a1_B & b\lambda^{-1}\\
c\lambda & d1_{\widehat B}
\end{pmatrix}.
\end{equation}

\begin{lemma}
\label{lem:scalar-embedding-cocycle}
Let $F=\bQ$ or $\bR$. For every $h_1,h_2\in\SL_2(F)$, one has
$$
c_B(\iota_\lambda(h_1),\iota_\lambda(h_2))
=\dim(B)\,\kappa(h_1,h_2).
$$
Consequently, for every $h\in\SL_2(F)$,
$$
\overline\tau_B(\iota_\lambda(h))
=\dim(B)\,\rot(h)  \qquad \text{in } \bR/\bZ.
$$
\end{lemma}

\begin{proof}
For $b_1b_2\neq0$, the symmetric homomorphism in (\ref{eqn:Orlov-cocycle-formula}) is $(b_{12}/(b_1b_2))\lambda$.
Its positive index is $\dim B$ when the scalar is positive and zero otherwise.
This proves the cocycle identity on the big cells, and Lemma~\ref{lem:big-subset-extension} extends it to all pairs.
Lemma~\ref{lem:SL2-coycle-rotation} then gives the homogenized formula.
\end{proof}

We next record the finite index reductions.

\begin{lemma}
\label{lem:basic-finite-index-finiteness}
Let $r\colon G\to \bR/\bZ$ satisfy $r(g^n)=nr(g)$ for $n\geq1$, and $H\subseteq G$ have finite index.
Then $r(G)$ is finite if and only if $r(H)$ is finite.
\end{lemma}

\begin{proof}
One implication is obvious, since $r(H)\subseteq r(G)$.
Conversely, let $d=[G:H]$ and $N=d!$. 
Then $g^N\in H$ for every $g\in G$. Thus
$$
r(G)\subseteq\{x\in\bR/\bZ:Nx\in r(H)\}.
$$
This proves the other implication.
\end{proof}

\begin{lemma}
\label{lem:isogeny-U(A)-U(B)-commensurable}
For an isogeny $\mu\colon B\to A$, the subgroups $U(B)$ and $f_\mu^{-1}U(A)f_\mu$ of $U(B)_\bQ$ are commensurable.
\end{lemma}

\begin{proof}
Choose $N\geq1$ so that $Nf_\mu$ and $Nf_\mu^{-1}$ are actual homomorphisms. Put
$$
R_X=\End(X\times\widehat{X}), \qquad \Gamma_X(N^2)=U(X)\cap(1+N^2R_X).
$$
For $u=1+N^2r\in\Gamma_A(N^2)$,
$$
f_\mu^{-1} u f_\mu = 1+(Nf_\mu^{-1})r(Nf_\mu) \in R_B.
$$
It lies in $U(B)_\bQ$, hence in $U(B)$. Similarly, $f_\mu \Gamma_B(N^2)f_\mu^{-1}\subseteq U(A)$.
The commensurability then follows from the fact that $\Gamma_X(N^2)$ has finite index in $U(X)$.
\end{proof}

\begin{cor}
\label{cor:isogeny-finite-invariant}
Finiteness of the reduced shifting spectrum is invariant under isogeny.
\end{cor}

\begin{proof}
This follows from Lemmas~\ref{lem:isogeny-cocycle}, \ref{lem:basic-finite-index-finiteness}, and \ref{lem:isogeny-U(A)-U(B)-commensurable}.
\end{proof}

%--------------------%--------------------%--------------------
\subsection{The finite spectrum case}
\label{subsec:6.2}
Let $B$ be a simple abelian variety with $\rho(B)=1$. Fix a polarization $\lambda\colon B\to\widehat{B}$, and choose $m\geq1$ such that $m\lambda^{-1}\colon\widehat{B}\to B$ is an actual morphism.
The congruence subgroup
$$
\Gamma^0(m)\coloneqq\left\{
\begin{pmatrix}
    a&b\\c&d
\end{pmatrix}
\in\SL_2(\bZ):
b\equiv0 \text{ (mod }m)
\right\}
$$
embeds into $U(B)$ under (\ref{eqn:scalar-embedding}).

\begin{prop}
\label{prop:picard-1-finite-index}
The subgroup $\iota_\lambda(\Gamma^0(m))$ has finite index in $U(B)$.
\end{prop}

\begin{proof}
Since $\rho(B)=1$, the polarization $\lambda$ identifies the complexified ample cone with the upper half-plane:
$$
\bH\cong D_B=\NS(B)_\bR+i\Amp(B), \qquad z\mapsto z\lambda.
$$
The natural action of $U(B)_\bR$ on $D_B$ induces a homomorphism of Lie groups
$$
\pi_B\colon U(B)_\bR
\to
\Aut_{\mathrm{hol}}(\bH)
\cong\PSL_2(\bR).
$$
For $h=\begin{pmatrix}a&b\\c&d\end{pmatrix}\in\SL_2(\bR)$,
the scalar embedding satisfies
\[
\iota_\lambda(h)\cdot(z\lambda)
=
\frac{c+dz}{a+bz}\lambda.
\]
Thus, under the standard identification of
$\operatorname{Aut}_{\mathrm{hol}}(\bH)$ with $\PSL_2(\bR)$,
\[
q(h)\coloneqq\pi_B(\iota_\lambda(h))
=
\left[
\begin{pmatrix}
d&c\\
b&a
\end{pmatrix}
\right].
\]
In particular, $q$ is a two-sheeted covering homomorphism, and $\pi_B$ is surjective.

By \cite{GLO}*{Theorem~8.2}, the stabilizer of $i\lambda\in D_B$ in $U(B)_\bR$ is compact.
The kernel $K\coloneqq\ker\pi_B$ is a closed subgroup of this stabilizer, hence is compact.
Moreover, $\pi_B$ is proper.
Indeed, for every compact subset $C\subseteq\PSL_2(\bR)$,
the set $q^{-1}(C)$ is compact and
$$
\pi_B^{-1}(C)
=
K\,\iota_\lambda\bigl(q^{-1}(C)\bigr),
$$
which is compact.

The faithful homology representation identifies $U(B)$ with a closed discrete subgroup of $U(B)_\bR$, since its elements preserve the lattice
$H_1(B\times\widehat B,\bZ)\cong\bZ^{4g}$.
Since $\pi_B$ is proper, its image $\Delta\coloneqq\pi_B(U(B))$ is discrete, and $F\coloneqq U(B)\cap\ker\pi_B$ is finite.

The formula for $q$ identifies
$$
\Lambda\coloneqq\pi_B\bigl(\iota_\lambda(\Gamma^0(m))\bigr)
$$
with the image in $\PSL_2(\bR)$ of
$$
\Gamma_0(m)
=
\left\{
\begin{pmatrix}a&b\\c&d\end{pmatrix}\in\SL_2(\bZ)
:\ c\equiv0\pmod m
\right\}.
$$
Hence $\Lambda$ is a finite-index subgroup of $\PSL_2(\bZ)$ and is therefore a lattice in $\PSL_2(\bR)$.
Since $\Delta$ is a discrete subgroup of $\PSL_2(\bR)$ containing $\Lambda$,
the standard finite-index property of lattices gives $[\Delta:\Lambda]<\infty$.
Consequently,
$$
[U(B):\iota_\lambda(\Gamma^0(m))] \leq |F|\cdot [\Delta:\Lambda] < \infty.
$$
\end{proof}

\begin{thm}
\label{thm:condition-star-satisfy}
If $A$ satisfies Condition~($\star$), then the reduced shifting spectrum $\Sigma^\red(A)$ is finite.
\end{thm}

\begin{proof}
For a simple factor $B$ of Picard number one, Lemmas~\ref{lem:SL2-coycle-rotation} and \ref{lem:scalar-embedding-cocycle} give a finite reduced spectrum on $\iota_\lambda(\Gamma^0(m))$.
Proposition~\ref{prop:picard-1-finite-index} and Lemma~\ref{lem:basic-finite-index-finiteness} give finiteness on $U(B)$.

Let $A$ be an abelian variety satisfying Condition~($\star$), i.e.~it is isogenous to $B=\prod_i B_i$, where the $B_i$'s are pairwise non-isogenous simple abelian varieties of Picard number one.
Then homomorphisms between distinct factors and their duals vanish, so $U(B)\cong\prod_i U(B_i)$.
Then, by Lemma~\ref{lem:product-cocycle}, $\Sigma^\red(B)$ is a sum of finitely many finite sets, therefore finite.
The theorem then follows from Corollary~\ref{cor:isogeny-finite-invariant}.
\end{proof}

%--------------------%--------------------%--------------------
\subsection{A Rosati-fixed field on a polarized factor}
\label{subsec:6.3}

Let $(B,\lambda_B)$ be a polarized abelian variety, and let $L\subseteq\End^0(B)^\dagger$ be a Rosati-fixed number field.
Then there is an embedding
\begin{equation}
\label{eqn:SL2(L)-U(B)}
\SL_2(L)\hookrightarrow U(B)_\bQ, \qquad
h=\begin{pmatrix}
    a & b \\ c & d 
\end{pmatrix}
\mapsto
T_h=\begin{pmatrix}
    a & b \lambda_B^{-1} \\ \lambda_B c & \lambda_B d \lambda_B^{-1}
\end{pmatrix}.
\end{equation}
When $L=\bQ$, this recovers the scalar embedding (\ref{eqn:scalar-embedding}).

The $L$-vector space $V=H_1(B,\bQ)$ gives a decomoposition
$$
V_\bR=\bigoplus_{\sigma\colon L\hookrightarrow\bR} V_\sigma, \qquad
V_\sigma=V\otimes_{L,\sigma}\bR.
$$
The complex structure preserves each summand;
we have $\dim_\bC V_\sigma=\frac{\dim B}{[L:\bQ]}\in\bZ_{>0}$.

\begin{lemma}
\label{lem:pb-lambda-v}
For every $v\in L\otimes_\bQ \bR$, $\lambda_Bv$ is real symmetric homomorphism, and
$$
p_B(\lambda_B v)=\frac{\dim B}{[L:\bQ]}\#\{\sigma\colon L\hookrightarrow\bR \ |\  \sigma(v)>0\}.
$$
\end{lemma}

\begin{proof}
It follows from $L$ is Rosati-fixed that $\lambda_Bv$ is symmetric.
The associated Hermitian form is
$$
H_{\lambda_B v}(w_1,w_2)=H_{\lambda_B}((\lambda_B^{-1}\lambda_Bv)w_1,w_2)=H_{\lambda_B}(vw_1,w_2).
$$
Hence, on $V_\sigma$ the Hermitian form $H_{\lambda_Bv}$ is given by $\sigma(v) H_{\lambda_B}|_{V_\sigma}$.
The lemma then follows from the fact that $H_{\lambda_B}$ is positive definite and that $V_\sigma$ has complex dimension $\frac{\dim B}{[L:\bQ]}$.
\end{proof}

The following proposition is a generalization of Lemma~\ref{lem:scalar-embedding-cocycle}.

\begin{prop}
\label{prop:SL2(L)-cocycle}
For $h_1,h_2\in\SL_2(L)$,
$$
c_B(T_{h_1},T_{h_2})=\frac{\dim B}{[L:\bQ]}\sum_{\sigma\colon L\hookrightarrow\bR} \kappa(\sigma(h_1),\sigma(h_2)).
$$
Hence
$$
\overline{\tau}_B(T_h)=\frac{\dim B}{[L:\bQ]}\sum_{\sigma\colon L\hookrightarrow\bR} \rot(\sigma(h)) \qquad \text{in } \bR/\bZ.
$$
\end{prop}

\begin{proof}
Work first on $\prod_\sigma\SL_2(\bR)$ with two inputs whose upper-right entries $b_1,b_2$ are nonzero. By (\ref{eqn:Orlov-cocycle-formula}) and Lemma~\ref{lem:pb-lambda-v},
\begin{align*}
    c_B(T_{h_1},T_{h_2}) &=  p_B(\lambda_B b_1^{-1}b_{12}b_2^{-1})-\dim(B) \\
    & = \frac{\dim B}{[L:\bQ]} \# \{\sigma\colon L\hookrightarrow\bR \ | \ \sigma(b_1^{-1}b_{12}b_2^{-1})>0\}-\dim(B) \\
    & = \frac{\dim B}{[L:\bQ]} \sum_{\sigma\colon L\hookrightarrow\bR} \kappa(\sigma(h_1),\sigma(h_2)).
\end{align*}
Lemma~\ref{lem:big-subset-extension} extends this identity to all pairs.
Homogenization and Lemma~\ref{lem:SL2-coycle-rotation} prove the second equality.
\end{proof}

\begin{lemma}
\label{lemma:theta-pi-irrational}
Suppose $\beta=2\cos\theta$ is algebraic, with $0<\theta<\pi$, and has a conjugate $\beta'$ with $|\beta'|>2$.
Then $\theta/\pi$ is transcendental.
\end{lemma}

\begin{proof}
If $\theta/\pi$ were rational, then $\alpha=e^{i\theta}$ would be a root of unity.
Every conjugate of $\beta=\alpha+\alpha^{-1}$ would then be the sum of a root of unity and its inverse, and hence would lie in $[-2,2]$.
Thus, $\theta/\pi$ is irrational.
Since $\alpha=e^{\pi i(\theta/\pi)}$ is algebraic, the Gelfond--Schneider theorem implies that the irrational number $\theta/\pi$ cannot be algebraic; hence it is transcendental.
\end{proof}

\begin{prop}
\label{prop:tauB-irrational}
Suppose $[L:\bQ]>1$ and $h\in\SL_2(L)$ has exactly one elliptic real embedding and other embeddings are hyperbolic:
$$
|\sigma_0(\tr h)|<2, \qquad |\sigma(\tr h)|>2 \quad (\sigma\neq\sigma_0).
$$
Then $\overline{\tau}_B(T_h)$ is transcendental.
\end{prop}

\begin{proof}
Write $\sigma_0(\tr h)=2\cos\theta$ with $0<\theta<\pi$.
Any other embedding gives a conjugate outside $[-2,2]$, so $\theta/\pi$ is transcendental by Lemma~\ref{lemma:theta-pi-irrational}.
The hyperbolic factos have zero rotation, while the elliptic factor has rotation $\pm\theta/\pi$.
The conclusion then follows from Proposition~\ref{prop:SL2(L)-cocycle}.
\end{proof}

Proposition~\ref{prop:tauB-irrational}, together with Proposition~\ref{prop:Rosati-fixed-condition-star-failed} below, will be used to construct autoequivalences with transcendental shifting numbers when Condition~$(\star)$ fails.

\begin{lemma}
\label{lem:Rosati-fixed-field-picard>1}
If $B$ is simple and $\rho(B)>1$, then there is a number field $\bQ\subsetneq L\subseteq\End^0(B)^\dagger$.
\end{lemma}

\begin{proof}
Recall that $\NS(B)_\bQ\cong\End^0(B)^\dagger$, so $\rho(B)>1$ implies that there exists a non-scalar element $x\in\End^0(B)^\dagger$.
Since $B$ is simple, $\End^0(B)$ is a division algebra.
Then $L\coloneqq\bQ[x]$ is a field that strictly contains $\bQ$, and is contained in $\End^0(B)^\dagger$ since every polynomial in $x$ is self-adjoint.
\end{proof}

\begin{lemma}
\label{lem:Rosati-fixed-field-two-copies}
Let $B_1,B_2$ be isogenous simple abelian varieties of Picard number one, with polarizations $\lambda_i$.
For the product polarization on $B=B_1\times B_2$, there is a number field $\bQ\subsetneq L\subseteq\End^0(B)^\dagger$.
\end{lemma}

\begin{proof}
Choose an isogeny $f\colon B_1\to B_2$, and let $f^\dagger=\lambda_1^{-1}\widehat{f}\lambda_2$.
Then $f^\dagger f$ is positive and self-adjoint.
Since $\End^0(B_1)^{\dagger_1}=\bQ1_{B_1}$, there is $q\in\bQ_{>0}$ such that $f^\dagger f=q1_{B_1}$, and hence $ff^\dagger=q1_{B_2}$.
For $t\in\bQ$, the endomorphism
$$
x_t=
\begin{pmatrix}
    0 & f^\dagger \\ f & t1_{B_2}
\end{pmatrix}
\in\End^0(B)
$$
is self-adjoint and satisfies $x_t^2-tx_t-q1_B=0$.
Choose $t\in\bQ$ so that $t^2+4q$ is not a rational square.
Then $L\coloneqq\bQ[x_t]$ has the desired property.
\end{proof}

\begin{prop}
\label{prop:Rosati-fixed-condition-star-failed}
Suppose $A$ fails Condition~($\star$), and fix a polarization $\lambda_A$.
There are polarized abelian varieties $B,C$, a polarization-orthogonal isogeny
$$
\mu\colon B\times C\to A, \qquad \widehat{\mu}\lambda_A\mu=\lambda_B\oplus\lambda_C,
$$
and a totally real number field $\bQ\subsetneq L\subseteq \End^0(B)^{\dagger_B}$.
Here $C$ may be zero.
\end{prop}

\begin{proof}
Take an orthogonal Poincar\'e decomposition of $A$ into simple factors.
If a factor has Picard number greater than one, take it as $B$ and apply Lemma~\ref{lem:Rosati-fixed-field-picard>1}.
Otherwise all factors have Picard number one, and failure of Condition~($\star$) means that two are isogenous.
Take their product as $B$ and apply Lemma~\ref{lem:Rosati-fixed-field-two-copies}.
Let $C$ be the product of the remaining factors.
The totally reality of $L$ follows from Lemma~\ref{lem:totally-real-L}.
\end{proof}

%--------------------%--------------------%--------------------
\subsection{Transcendental shifting numbers}
\label{subsec:6.4}
In this subsection, we complete the proofs of Theorem~\ref{thm:main-shifting-spectrum} and Corollary~\ref{cor:main-DHKK-algebraicity}.

Keep the data of Proposition~\ref{prop:Rosati-fixed-condition-star-failed}.
Put
$$
i_B=\mu|_{B\times\{0\}}, \qquad 
p_B=\text{pr}_B\mu^{-1}, \qquad
e=i_Bp_B.
$$
Then $p_Bi_B=1_B$, and
$$
j\colon \End^0(B)\to\End^0(A), \qquad
j(a)\coloneqq i_B a p_B
$$
is an injective multiplicative map with $j(1_B)=e$.

\begin{lemma}
\label{lem:uA(h)}
The homomorphism $\SL_2(L)\to U(A)_\bQ$ obtained from (\ref{eqn:SL2(L)-U(B)}) by extension by the identity on $C\times\widehat{C}$ and conjugation by $f_\mu$, sends
$$
h=
\begin{pmatrix}
    1+a & b \\ c & 1+d 
\end{pmatrix}
$$
to 
$$
u_A(h)=
\begin{pmatrix}
    1_A+j(a) & j(b)\lambda_A^{-1} \\
    \lambda_A j(c) & 1_{\widehat{A}}+\lambda_A j(d) \lambda_A^{-1}
\end{pmatrix}.
$$
Moreover, $\overline{\tau}_A(u_A(h))=\overline{\tau}_B(T_h)$.
\end{lemma}

\begin{proof}
The map $h\mapsto u_A(h)$ follows from a straightforward calculation.
Lemmas~\ref{lem:isogeny-cocycle} and \ref{lem:product-cocycle} give the equality of the reduced shifting numbers.
\end{proof}

\begin{lemma}
\label{lem:prep-for-transcendental}
Let $L\supsetneq\bQ$ be a totally real field, let $\eta$ be a primitive element, and let $M\geq1$ be an integer.
There are an embedding $\sigma_0\colon L\hookrightarrow\bR$, an integer $n>0$, and $q\in\bZ$ such that
$$
-\frac{4}{M^2}<n\sigma_0(\eta)+q<0, \qquad n\sigma(\eta)+q>0 \quad (\sigma\neq\sigma_0).
$$
\end{lemma}

\begin{proof}
The conjugates of $\eta$ are distinct; choose $\sigma_0$ where $\sigma_0(\eta)$ attain their minimum. 
Note that $\sigma_0(\eta)$ is irrational.
Density of its positive multiples modulo one gives arbitrarily large $n$, with suitable $q\in\bZ$, such that
$$
-\frac{4}{M^2}<n\sigma_0(\eta)+q<0.
$$
For $\sigma\neq\sigma_0$,
$$
n\sigma(\eta)+q=n(\sigma(\eta)-\sigma_0(\eta))+(n\sigma_0(\eta)+q).
$$
The first term tends to infinity and the second is bounded.
Choose $n$ sufficiently large for all the other embeddings.
\end{proof}

\begin{thm}
\label{thm:condition-star-fail}
If $A$ fails Condition~($\star$), there is an autoequivalence of $\Db(A)$ with transcendental shifting number.
\end{thm}

\begin{proof}
Choose the data of Proposition~\ref{prop:Rosati-fixed-condition-star-failed} and a primitve element $\eta\in L$.
Fix a positive integer $M$ so that
\begin{equation}
\label{eqn:clear-denominator}
Mj(\eta),Me\in\End(A), \qquad Me\lambda_A^{-1}\in\Hom(\widehat{A},A).
\end{equation}
Choose $\sigma_0,n,q$ from Lemma~\ref{lem:prep-for-transcendental} and set
$$
x=M(n\eta+q), \qquad
h=
\begin{pmatrix}
    1+Mx & M \\ x & 1 
\end{pmatrix}
\in \SL_2(L).
$$
Its traces satisfy
$$
|2+M\sigma_0(x)|<2, \qquad 2+M\sigma(x)>2 \quad (\sigma\neq\sigma_0).
$$
By Lemma~\ref{lem:uA(h)}, its image in $U(A)_\bQ$ is
$$
u_A(h)=
\begin{pmatrix}
    1_A+Mj(x) & Me\lambda_A^{-1} \\
    \lambda_A j(x) & 1_{\widehat{A}}
\end{pmatrix}.
$$
In fact, all the entries are integral by (\ref{eqn:clear-denominator}).
Thus $u_A(h)\in U(A)$.

By Orlov's surjectivity, there exists $\Phi\in\Aut\Db(A)$ such that $u_\Phi=u_A(h)$. We have $\overline{\tau}(\Phi)=\overline{\tau}_A(u_A(h))$.
It is transcendental by Lemma~\ref{lem:uA(h)} and Proposition~\ref{prop:tauB-irrational}.
\end{proof}

We can now prove Theorem~\ref{thm:main-shifting-spectrum}.

\begin{proof}[Proof of Theorem~\ref{thm:main-shifting-spectrum}]
Bauer's theorem \cite{Bauer} identifies Condition~($\star$) with rational polyhedrality of $\Nef(A)$.
The remaining equivalences follow from Theorem~\ref{thm:condition-star-satisfy} and Theorem~\ref{thm:condition-star-fail}.
\end{proof}

\begin{lemma}
\label{lemma:entropy-graph-criterion}
Suppose $\Phi\in\Aut\Db(A)$ has $\tau(\Phi)\notin\bQ$.
Then there is no nonzero polynomial $P\in\bC[x,y]$ satisfying
$$
P(e^t,e^{h_t(\Phi)})=0 \qquad (t\in\bR).
$$
\end{lemma}

\begin{proof}
By Theorem~\ref{thm:main-entropy}, $h_t(\Phi)=\log\lambda+bt$ is linear, where $\lambda=\rho(N(\Phi))>0$ and $b=\tau(\Phi)\notin\bQ$.
Suppose such a polynomial $P(x,y)=\sum_{i,j}c_{i,j}x^iy^j$ exists. Then
$$
\sum_{i,j} c_{ij}\lambda^j e^{(i+jb)t}=0 \qquad (t\in\bR).
$$
Since $b\notin\bQ$, the exponents $i+jb$ are distinct for distinct pairs $(i,j)$.
The exponential functions $e^{\alpha t}$ with distinct real exponents
$\alpha$ are linearly independent over $\bC$, thus every $c_{ij}$ is zero.
\end{proof}

\begin{proof}[Proof of Corollary~\ref{cor:main-DHKK-algebraicity}]
Apply Lemma~\ref{lemma:entropy-graph-criterion} to any autoequivalence from Theorem~\ref{thm:condition-star-fail}.
\end{proof}

\bigskip
\bibliography{ref}
\bibliographystyle{alpha}

\ \\

\small{
\noindent Yu-Wei Fan \\
\textsc{Center for Mathematics and Interdisciplinary Sciences, Fudan University, \\ Shanghai 200433, China}\\
\textsc{Shanghai Institute for Mathematics and Interdisciplinary Sciences (SIMIS), \\  Shanghai 200433, China}\\
\texttt{yuweifanx@gmail.com}
}

\end{document}